\documentclass[final]{siamart251216}

\usepackage[utf8]{inputenc}
\usepackage[T1]{fontenc}
\usepackage[english]{babel}
\usepackage{amssymb,mathtools}   
\usepackage[numbers,sort&compress]{natbib}
\usepackage{booktabs}
\usepackage{array}
\usepackage{enumitem}
\usepackage{listings}
\usepackage{subcaption}

\headers{Threshold Dynamics of Voter Radicalization}{A. Omelchenko}

\newcommand{\R}{\mathbb{R}}
\newcommand{\Simp}{\Sigma_2}   

\newsiamremark{remark}{Remark}

\title{Threshold Dynamics of Voter Radicalization on the Probability Simplex}

\author{Alexander Omelchenko%
\thanks{Constructor University Bremen gGmbH,
Campus Ring 1, 28759 Bremen, Germany
(\email{aomelchenko@constructor.university}).}}

\begin{document}
\maketitle

\begin{abstract}
We analyse two coupled ODE models of political competition on invariant
probability simplices with a conserved electorate. The baseline three-group
model tracks left-radical, centrist, and right-radical voter shares. We
characterise the unique interior equilibrium by a Perron--Frobenius threshold,
establish global asymptotic stability in the symmetric and asymmetric cases,
and exclude periodic orbits unconditionally via the Dulac criterion. A
structural consequence is that the baseline model cannot produce irreversible
centrist decline, history-dependent long-run floors, or multiple attractors.
We then extend the model with a disengaged voter compartment and distinguish
pure state shocks from permanent structural parameter shifts. The post-shock
dynamics are governed by the same spectral threshold: below it the centrist
state is globally asymptotically stable; above it every trajectory with a
nonzero radical seed converges to the unique radicalised equilibrium.
Cumulative sub-threshold structural shifts can cross the threshold and produce
staircase dynamics absent from the baseline; the symmetric reduction yields
closed-form expressions for the critical shock amplitude and the radicalization
window.
\end{abstract}



\section{Introduction}

\subsection{Empirical background}

Across Western democracies, radical movements on both the left and the right
have grown while the mainstream centre has weakened.  Comparative evidence for
this broad trend appears in studies of populist mobilisation, ideological
sorting, and party-system polarisation \cite{mudde2019,pew2014,dalton2021,norris2019}.
A recurring empirical regularity is \emph{asymmetric mobility}: voters move
mainly between the centre and a radical camp, but only rarely directly between
the two radical wings.  This motivates a minimal three-group representation
with left-radical, centrist, and right-radical shares.

The post-2008 period adds a second feature that a baseline model should either
explain or formally fail to explain: repeated crises are often followed not by
full recovery but by a higher post-crisis floor of radical support.  Germany
and France provide prominent examples.  In Germany, the AfD moved from 4.7\%
in 2013 to 12.6\% in 2017, 10.3\% in 2021, and 20.8\% in the 2025 federal
election \cite{bundeswahlleiterin2025}.  France provides a second illustration: both a left-wing radical bloc
(LFI) and a right-wing radical bloc (RN) strengthened substantially
after 2012 while mainstream parties lost ground~\cite{ministere2024}.
Because presidential, legislative, and European elections operate under
different electoral rules, we treat French data only as qualitative
motivation and not as a homogeneous time series.  The qualitative
feature of interest is not any single election result but the apparent ratchet
pattern: after successive crises, radical support tends to settle at a higher
floor.

A conceptual gap underlies much of the existing discussion.  Most
commentary on European radicalisation asks \emph{how much} radical
support has grown, but not \emph{whether} that growth is reversible.
The two questions require different analyses.  A transient surge driven
by crisis disengagement can dissipate as the crisis fades, even if it
looks alarming at its peak.  A structural shift in the recruitment
environment---lower institutional trust, higher salience of
identity-based grievances, weakened centrist party organisations---may
persist indefinitely regardless of whether the triggering crisis is
resolved. To our knowledge, no existing model within the compartmental
conserved-electorate ODE framework combines reversible state
shocks, permanent structural parameter shifts, and a full
qualitative analysis on the simplex. The present paper proposes one such formulation.

This paper develops a hierarchy of ODE models on the probability simplex to
separate these two questions.  The baseline model captures normal
self-stabilising competition under a conserved electorate.  The extended
four-group model adds a disengaged compartment and distinguishes state shocks
from structural shocks, thereby allowing irreversible shifts and staircase
dynamics.  We stress at the outset that this is a theorem-led stylised model:
our goal is to identify the minimal structural conditions under which
radicalization can be irreversible, not to provide a calibrated causal account
of specific elections.

\subsection{Related modelling literature}

Bounded-confidence and network models describe individual opinion formation and
cluster formation \cite{deffuant2000,hegselmann2002,lorenz2007,baumann2020}, but
aggregate vote-share dynamics are harder to analyse sharply in that setting.
Aggregate ODE models closer in spirit to the present paper include the
satisficing two-party model of \cite{yang2020}, which tracks ideological
\emph{positions} of parties rather than voter shares, and the three-group
sociophysics model of \cite{diep2024}, applied to the 2024 US election.

The most directly comparable line of work uses compartmental models
inspired by epidemiological SIS dynamics to track voter-share fractions.
Volkening, Linder, Porter, and Rempala \cite{volkening2020} introduced the
cRUD framework for two-party US elections: Democratic ($D$), Republican ($R$),
and undecided ($U$) fractions satisfy a conserved-electorate ODE on the
probability simplex, with parameters fitted to state-level polling data and
10\,000 stochastic simulations used to produce probabilistic forecasts.
Branstetter et al.\ \cite{branstetter2024} extended that approach to
presidential, senatorial, and gubernatorial races from 2004 to 2022, also
published in this journal.

The present paper shares the simplex constraint and the conserved-electorate
assumption with the Volkening--Branstetter line, but differs from it in
question, scope, and method in four respects.

\textit{Question.}  The cRUD work asks who wins a given election
(a forecasting question answered by parameter fitting and Monte Carlo
simulation).  The present paper asks whether the system has a stable
centrist equilibrium, what bifurcations can disrupt it, and whether crises
can permanently shift the long-run attractor (a structural-stability question
answered by proof).

\textit{Model.}  The cRUD model has two decided compartments ($D$, $R$) with
undecided voters as the mediating pool.  The present model has three to four
compartments with cross-wing \emph{reactive polarisation} terms
$\gamma_{LR}$ and $\gamma_{RL}$: a voter moving to the right radical camp
directly amplifies left-radical recruitment and vice versa, a feedback absent
from SIS-type two-party models.

\textit{Geography and party system.}  The cRUD line is calibrated to the
US two-party system.  The present model targets multi-party European systems
where a radical left and a radical right can grow simultaneously while the
mainstream centre weakens---the configuration most visible in France
(simultaneous RN and LFI growth) and Germany (AfD alongside persistent
left-wing support).

\textit{Mathematical completeness.}  The cRUD line pursues a forecasting
goal and does not aim at qualitative theory; accordingly, global asymptotic
stability, bifurcation conditions, and exclusion of periodic orbits lie
outside its scope.  The present paper proves all three.  In particular, the
Dulac function $B=(LRC)^{-1}$ yields a strictly negative divergence on the
entire interior of the simplex for \emph{all} positive parameters---a result
that, to our knowledge, does not appear in the compartmental-election
literature.  In the four-group extension we further
prove that the distinction between reversible surges and irreversible political
shifts is governed by a spectral threshold, which in the symmetric case
reduces to the explicit formula $\Delta_c = (\mu-\beta)/(\delta-\beta)$.

The paper also connects to the literature on dynamical systems on simplices and
monotonicity-type arguments \cite{hirsch1988,zeeman1993}.  Our model is
not competitive in the strict Hirsch sense throughout $\Simp$, because the
off-diagonal Jacobian entries change sign across the interior.  The global
results therefore combine a Perron--Frobenius threshold, a positive-vector
Lyapunov argument, and the Dulac exclusion, rather than monotone-systems
theory.

A separate literature studies democratic erosion and institutional backsliding
\cite{haggard2021,grillo2024,torgler2024,abels2024}.  The extended model
below is closer to that literature in spirit, but remains explicitly tied to
voter flows between political compartments.

\subsection{Contribution and positioning}

The paper makes two main contributions.

First, for the baseline three-group model we give a qualitative analysis on
the simplex: forward invariance; global asymptotic stability in the symmetric
and asymmetric cases; an explicit equilibrium formula $C^*=\mu/(\alpha+\gamma)$
in the symmetric reduction; a Perron--Frobenius threshold in the asymmetric
case; a transcritical bifurcation at $\beta=\mu$; and unconditional exclusion
of periodic orbits via the Dulac function $B=(LRC)^{-1}$.  A structural
consequence is that the baseline model cannot generate staircase dynamics or
multiple attractors under fixed parameters.

Second, we introduce a four-group disengagement-shock extension with a
complete global classification: if $\mathcal R_{\mathrm{rad}}\le1$, the
centrist state is globally asymptotically stable; if $\mathcal
R_{\mathrm{rad}}>1$, every trajectory with a nonzero radical seed converges to
the unique radicalised equilibrium (Theorem~\ref{thm:global4}).  We further
derive a spectral shock threshold $\Delta_c^{\mathrm{asym}}$
(Theorem~\ref{thm:PF_Delta}), a Perron-vector window bound $t_q^*$
(Proposition~\ref{prop:PF_window}), and closed-form symmetric reductions for
the critical shock amplitude and radicalization window.

The threshold $\mathcal R_{\mathrm{rad}}=\lambda_{\mathrm{PF}}(M^{-1}K)$ is
formally analogous to the next-generation matrix of
\cite{diekmann1990,vandendriessche2002}.  The present paper goes beyond that
framework in three respects: the Dulac function $B=(LRC)^{-1}$ unconditionally
excludes periodic orbits on the full simplex interior; global asymptotic
stability is established on the entire simplex rather than locally; and the
four-group extension analyses threshold \emph{crossing} via permanent
structural shocks---questions outside the scope of the NGM apparatus.

The paper's primary aim is theorem-led: to identify the minimal structural
conditions under which radicalization is irreversible, not to provide a
calibrated causal account of specific elections.  The empirical illustrations
are qualitative only.

The paper is organised as follows.  
Sections~\ref{sec:model}--\ref{sec:numerics} develop and analyse the 
baseline model.  Section~\ref{sec:model4} introduces the four-group 
extension.  Section~\ref{sec:analysis4} provides the mathematical analysis.  
Section~\ref{sec:cumulative} derives the shock and cumulative-drift results.  
Section~\ref{sec:empirical} provides qualitative illustrations against German 
and French data.  Section~\ref{sec:numerics4} gives numerical illustrations.

\section{Baseline Model Formulation}\label{sec:model}

\subsection{State space and governing equations}

Let $L(t)$, $R(t)$, and $C(t)$ denote the left-radical, right-radical, and
centrist shares of a normalised electorate, with
\[
  L(t)\ge0,\qquad R(t)\ge0,\qquad C(t)=1-L(t)-R(t)\ge0.
\]
The feasible state space is the standard simplex
\[
  \Simp = \{(L,R)\in\R^2 : L\ge0,\;R\ge0,\;L+R\le1\}.
\]

\begin{definition}[The baseline model]\label{def:model}
The \emph{normalised three-group model} is the autonomous system
\begin{equation}\label{eq:model}
\begin{aligned}
  \dot{L} &= \alpha_L\,L\,C \;-\; \mu_L\,L \;+
             \gamma_{RL}\,R\,C, \\
  \dot{R} &= \alpha_R\,R\,C \;-
             \mu_R\,R \;+
             \gamma_{LR}\,L\,C,
\end{aligned}
\end{equation}
where $C = 1-L-R$ and $\alpha_L,\alpha_R,\mu_L,\mu_R,\gamma_{RL},\gamma_{LR}>0$.
\end{definition}

By construction, $\dot C=-\dot L-\dot R$, so $L+R+C=1$ is conserved.  The six
parameters of the baseline model are summarised in Table~\ref{tab:params}.

\subsection{Interpretation}

The terms $\alpha_i X_i C$ represent direct recruitment of centrists by the
corresponding radical camp, $-\mu_i X_i$ represents deradicalisation back to
the centre, and $\gamma_{ij} X_j C$ represents reactive polarisation: growth
of one radical wing pushes some centrists toward the opposite wing.  The
bilinear factors are the minimal mass-action closure compatible with
conservation and saturation: each flow vanishes when either the recruiting
population or the recruitable pool is absent.

We use $\gamma_{RL}RC$ rather than a wing-to-wing transfer term such as $RL$.
The intended mechanism is not direct switching between radical camps, which is
empirically rare, but movement of centrists toward one wing in response to the
visibility of the other.  This choice also preserves the simplex geometry used
in the analysis below: since $\gamma_{RL}RC\ge0$ vanishes on
the faces $R=0$ and $C=0$, the boundary of $\Simp$ remains
forward invariant (Proposition~\ref{prop:inv}).

\begin{table}[ht]
\caption{Summary of baseline model parameters.}\label{tab:params}
\centering
\renewcommand{\arraystretch}{1.4}
\begin{tabular}{lll}
\toprule
Parameter & Interpretation & Range \\
\midrule
$\alpha_L,\alpha_R$ & Direct recruitment rate (centrists $\to$ radicals) & $(0,\infty)$ \\
$\mu_L,\mu_R$       & Deradicalisation rate (radicals $\to$ centrists)    & $(0,\infty)$ \\
$\gamma_{RL}$       & Reactive polarisation: right growth $\to$ left       & $(0,\infty)$ \\
$\gamma_{LR}$       & Reactive polarisation: left growth $\to$ right       & $(0,\infty)$ \\
\bottomrule
\end{tabular}
\end{table}

\section{Mathematical Analysis of the Baseline Model}\label{sec:analysis}

The analysis follows a standard qualitative-dynamics programme for simplex
ODEs: (i)~forward invariance, (ii)~equilibrium classification,
(iii)~global stability via comparison and Lyapunov arguments,
(iv)~linearisation with explicit spectral data, and
(v)~Bendixson--Dulac to rule out periodic orbits.
Steps (i)--(iv) are carried out first for the symmetric reduction, where the
triangular structure of the $(S,D)$ system allows complete closed-form results;
step (v) is then proved for the full asymmetric model without symmetry
restrictions.

\subsection{Forward invariance}

\begin{proposition}[Forward invariance of $\Simp$]\label{prop:inv}
  For any $(L_0,R_0)\in\Simp$ and any positive parameters, the
  solution $(L(t),R(t))$ of \eqref{eq:model} remains in $\Simp$ for all $t\ge0$.
\end{proposition}

\begin{proof}
  We verify that the vector field $\mathbf{F}=(F_1,F_2)$ points into $\Simp$ (or is
  tangent to it) at every point on the boundary $\partial\Simp$.

  \emph{Face $L=0$, $R\in[0,1]$:}
  $F_1\big|_{L=0} = \gamma_{RL}\,R\,(1-R)\ge0$.
  Hence $\dot{L}\ge0$: the trajectory cannot exit through this face.

  \emph{Face $R=0$, $L\in[0,1]$:}
  By symmetry, $F_2\big|_{R=0} = \gamma_{LR}\,L\,(1-L)\ge0$.

  \emph{Face $C=0$, i.e.\ $L+R=1$:}
  Here $C=0$, so $F_1\big|_{C=0}=-\mu_L\,L\le0$ and
  $F_2\big|_{C=0}=-\mu_R\,R\le0$. Therefore $\dot{L}+\dot{R}\le0$, which means
  $\dot{C}=-\dot{L}-\dot{R}\ge0$: the total radical share cannot increase beyond
  1, so trajectories cannot exit through this face either.

  Since the vector field points inward (or is tangent) on all three faces, and
  the system has a unique locally Lipschitz solution by the Picard--Lindel\"of
  theorem, the simplex $\Simp$ is positively invariant.   
\end{proof}

\begin{remark}
  The boundary behaviour on the face $C=0$ is particularly informative.  When
  the centrist pool is entirely depleted, both $\dot{L}$ and $\dot{R}$ are
  non-positive: the system decelerates rather than accelerating, and must
  return to the interior of $\Simp$.  This guarantees $C(t)>0$ for all finite $t$
  whenever $C(0)>0$.
\end{remark}

\subsection{Symmetric reduction and global dynamics}

\begin{definition}[Symmetric case]\label{def:sym}
  We say the model is \emph{symmetric} if
  $\alpha_L=\alpha_R=\alpha$,\; $\mu_L=\mu_R=\mu$,\;
  $\gamma_{LR}=\gamma_{RL}=\gamma$.
\end{definition}

Under symmetric parameters, introduce the \emph{total radical share}
$S:=L+R$ and the \emph{left--right imbalance} $D:=L-R$.  Since $C=1-S$, a
direct computation from \eqref{eq:model} yields the \emph{triangular} system
\begin{equation}\label{eq:SD_system}
  \dot{S} = S\bigl[(\alpha+\gamma)(1-S)-\mu\bigr],
  \qquad
  \dot{D} = D\bigl[(\alpha-\gamma)(1-S)-\mu\bigr].
\end{equation}
Setting $\beta:=\alpha+\gamma$, the $S$-equation is logistic-like with unique
positive equilibrium $S^*=1-\mu/\beta$ when $\beta>\mu$.  Once $S(t)$ is
determined, $D$ satisfies a linear ODE.  The diagonal $\{L=R\}$ is the
invariant subspace $D=0$; on it $S=2P$ and the $S$-equation reduces to
\begin{equation}\label{eq:scalar}
  \dot{P} = P[\beta(1-2P)-\mu].
\end{equation}

\begin{theorem}[Global dynamics of the symmetric model]\label{thm:global}
  \begin{enumerate}
    \item If $\beta\le\mu$: $E_0=(0,0)$ is globally asymptotically stable on
      $\Simp$.  Both radical wings vanish: $L(t),R(t)\to0$, $C(t)\to1$.
    \item If $\beta>\mu$: the unique interior equilibrium
      $E_1=(P^*,P^*)$, with $P^*=\tfrac12(1-\mu/\beta)$ and $C^*=\mu/\beta$,
      is globally asymptotically stable on $\Simp\setminus\{E_0\}$; $E_0$ is
      unstable.
  \end{enumerate}
\end{theorem}

\begin{proof}
  \textbf{Case $\beta\le\mu$:}
  $\dot{S}=S[\beta(1-S)-\mu]\le S(\beta-\mu)\le0$, so $S(t)\to0$.
  Since $|D|\le S$, also $D(t)\to0$, hence $L,R\to0$.

  \textbf{Case $\beta>\mu$:}
  The $S$-equation is logistic-like; for any $S(0)\in(0,1]$,
  $S(t)\to S^*=1-\mu/\beta$ monotonically.  The $D$-equation is linear:
  \[
    \dot D = k(t)\,D,
    \qquad
    k(t):=(\alpha-\gamma)(1-S(t))-\mu.
  \]
  Hence
  \[
    D(t)=D(0)\exp\!\left(\int_0^t k(s)\,ds\right).
  \]
  Since $S(t)\to S^*=1-\mu/\beta$, we have
  \[
    k(t)\to (\alpha-\gamma)\frac{\mu}{\beta}-\mu
         = -\frac{2\gamma\mu}{\beta}<0.
  \]
  Therefore there exist $T>0$ and $\eta>0$ such that
  $k(t)\le -\eta$ for all $t\ge T$.  For $t\ge T$,
  \[
    |D(t)|
    \le |D(T)|\exp\!\left(-\eta(t-T)\right),
  \]
  so $D(t)\to0$ exponentially.  Consequently
  \[
    L(t)=\frac{S(t)+D(t)}{2}\to \frac{S^*}{2}=P^*,
    \qquad
    R(t)=\frac{S(t)-D(t)}{2}\to \frac{S^*}{2}=P^*.
  \]
  Instability of $E_0$: its dominant eigenvalue is
  $\lambda_1=\beta-\mu>0$.   
\end{proof}

\begin{corollary}[Centrist share at equilibrium]\label{cor:Cstar}
  When $E_1$ exists ($\beta>\mu$), the centrist share at equilibrium is
  \begin{equation}\label{eq:Cstar}
    C^* = \frac{\mu}{\beta} = \frac{\mu}{\alpha+\gamma}\in(0,1).
  \end{equation}
  $C^*$ is strictly decreasing in $\alpha$ and $\gamma$, strictly increasing
  in $\mu$, and positive for all positive parameters.
\end{corollary}

\begin{proof}
  Direct substitution of $P^*$ into $C=1-2P^*$.   
\end{proof}

\subsection{Stability analysis: two-dimensional case}

We now return to the full asymmetric system~\eqref{eq:model}.

\begin{definition}[Interior equilibrium]\label{def:E1_2d}
  A point $(L^*,R^*)$ with $L^*>0$, $R^*>0$, $L^*+R^*<1$ is an
  \emph{interior equilibrium} if $F_1(L^*,R^*)=0$ and $F_2(L^*,R^*)=0$.
\end{definition}

Setting $F_1=F_2=0$ with $C^*=1-L^*-R^*$ and collecting terms:
\begin{equation}\label{eq:eq_system}
  C^*\!\bigl(\alpha_L\,L^* + \gamma_{RL}\,R^*\bigr) = \mu_L\,L^*, \qquad
  C^*\!\bigl(\alpha_R\,R^* + \gamma_{LR}\,L^*\bigr) = \mu_R\,R^*.
\end{equation}
Together with $L^*+R^*+C^*=1$, this is a closed $3\times3$ system for
$(L^*,R^*,C^*)$.  In the symmetric case $L^*=R^*=P^*$ both equations reduce
to $C^*(\alpha+\gamma)=\mu$, i.e.\ $C^*=\mu/\beta$, consistent with
Corollary~\ref{cor:Cstar}.

The Jacobian at a general point $(L,R)\in\mathrm{int}(\Simp)$:
\begin{equation}\label{eq:jacobian}
  J = \begin{pmatrix}
    \alpha_L(C-L) - \mu_L - \gamma_{RL}\,R &
    \gamma_{RL}(C-R) - \alpha_L\,L \\[4pt]
    \gamma_{LR}(C-L) - \alpha_R\,R &
    \alpha_R(C-R) - \mu_R - \gamma_{LR}\,L
  \end{pmatrix}.
\end{equation}
Note that $J_{11}\ne (\alpha_L+\gamma_{RL})(C-L)-\mu_L$: the $\gamma_{RL}$-term
differentiates differently through $L$ and $R$ because in $F_1$ it multiplies $R\,C$,
not $L\,C$.

\begin{theorem}[Existence, uniqueness, and stability of the interior
  equilibrium]\label{thm:PF}
Define the \emph{recruitment matrix} and \emph{decay matrix}:
  \[
    K:=\begin{pmatrix}\alpha_L & \gamma_{RL}\\\gamma_{LR} & \alpha_R\end{pmatrix},
    \quad
    M:=\begin{pmatrix}\mu_L & 0\\ 0 & \mu_R\end{pmatrix}.
  \]
Let $\lambda_{\mathrm{PF}}>0$ be the Perron root of $M^{-1}K$ and $u=(u_1,u_2)^\top\gg0$ the
corresponding positive eigenvector, with explicit formula
  \[
    \lambda_{\mathrm{PF}}=
    \frac{1}{2}\!\left(\frac{\alpha_L}{\mu_L}+\frac{\alpha_R}{\mu_R}
    +\sqrt{\!\left(\frac{\alpha_L}{\mu_L}-\frac{\alpha_R}{\mu_R}\right)^{\!2}
    +\frac{4\gamma_{RL}\gamma_{LR}}{\mu_L\mu_R}}\,\right).
  \]
  \begin{enumerate}
    \item An interior equilibrium exists if and only if $\lambda_{\mathrm{PF}}>1$.
    \item When it exists it is unique:
      \[
        C^*=\frac{1}{\lambda_{\mathrm{PF}}},\qquad
        (L^*,R^*)=\frac{1-C^*}{u_1+u_2}\,(u_1,u_2).
      \]
    \item It is locally asymptotically stable.  In the symmetric case it is a
      stable node (or stable star on a codimension-one locus,
      Theorem~\ref{thm:2d}).  In the asymmetric case it may be a stable node
      or a stable \emph{focus}.
  \end{enumerate}
\end{theorem}

\begin{proof}
  \textbf{Parts 1--2.}
 At an interior equilibrium, the equations \eqref{eq:eq_system} read
  $C^* Kv = Mv$ for $v=(L^*,R^*)^\top\gg0$.  Multiplying by $M^{-1}$:
  $(M^{-1}K)v=(1/C^*)v$.  By the Perron--Frobenius theorem, the only positive
  eigenvector of the positive matrix $M^{-1}K$ is the Perron eigenvector with
  eigenvalue $\lambda_{\mathrm{PF}}$, so $C^*=1/\lambda_{\mathrm{PF}}\in(0,1)$
  iff $\lambda_{\mathrm{PF}}>1$.  The explicit formula follows from the $2\times2$
  characteristic polynomial of $M^{-1}K$.

  In the symmetric case $\lambda_{\mathrm{PF}}=\beta/\mu$, so the condition
  $\lambda_{\mathrm{PF}}>1$ reduces exactly to $\beta>\mu$, and
  $C^*=\mu/\beta$, consistent with Corollary~\ref{cor:Cstar}.

  \textbf{Part 3 (local stability).}
  At an interior equilibrium, the relations
  \[
    \mu_L = C^*\!\left(\alpha_L+\gamma_{RL}\frac{R^*}{L^*}\right),
    \qquad
    \mu_R = C^*\!\left(\alpha_R+\gamma_{LR}\frac{L^*}{R^*}\right)
  \]
  follow from \eqref{eq:eq_system}.  Substituting them into the Jacobian
  \eqref{eq:jacobian} gives
  \[
    J_{11}
    = -\alpha_L L^* - \gamma_{RL}R^*
      - C^*\gamma_{RL}\frac{R^*}{L^*},
    \qquad
    J_{22}
    = -\alpha_R R^* - \gamma_{LR}L^*
      - C^*\gamma_{LR}\frac{L^*}{R^*}.
  \]
  Therefore
  \[
    \mathrm{tr}\,J
    = -\left[
      \alpha_L L^* + \alpha_R R^*
      + \gamma_{RL}R^* + \gamma_{LR}L^*
      + C^*\gamma_{RL}\frac{R^*}{L^*}
      + C^*\gamma_{LR}\frac{L^*}{R^*}
    \right] < 0.
  \]
For the determinant, expand $\det J = J_{11}J_{22}-J_{12}J_{21}$.
  Using the expressions for $J_{11}$, $J_{22}$ above, and
  \[
    J_{12} = \gamma_{RL}(C^*-R^*)-\alpha_L L^*,
    \qquad
    J_{21} = \gamma_{LR}(C^*-L^*)-\alpha_R R^*,
  \]
  together with the equilibrium relations $\mu_L = C^*(\alpha_L +
  \gamma_{RL}R^*/L^*)$ and $\mu_R = C^*(\alpha_R + \gamma_{LR}L^*/R^*)$,
  a direct computation gives
  \[
    \det J
    = C^*(L^*+R^*)
      \left(
        \alpha_L\gamma_{LR}\frac{L^*}{R^*}
        + 2\gamma_{LR}\gamma_{RL}
        + \alpha_R\gamma_{RL}\frac{R^*}{L^*}
      \right) > 0.
  \]
  Hence $\mathrm{tr}\,J < 0$ and $\det J > 0$, so the interior equilibrium
  is locally asymptotically stable. The
  node-vs-focus classification depends on the sign of the discriminant
  $(\mathrm{tr}\,J)^2-4\det J$.   
\end{proof}

\begin{theorem}[Stability in the symmetric 2D case]\label{thm:2d}
  In the symmetric model ($\beta>\mu$), the interior equilibrium $E_1=(P^*,P^*)$
  is locally asymptotically stable with eigenvalues
  \begin{equation}\label{eq:eigenvalues}
    \lambda_1 = \mu - \beta < 0, \qquad
    \lambda_2 = -\frac{2\gamma\mu}{\beta} < 0.
  \end{equation}
  Generically ($\lambda_1\ne\lambda_2$) it is a \emph{stable node}.  The two
  eigenvalues coincide on the codimension-one locus
  $2\gamma\mu=\beta(\beta-\mu)$, at which $E_1$ is a \emph{stable star}.
\end{theorem}

\begin{proof}
  Evaluate \eqref{eq:jacobian} at $L=R=P^*$ with $C^*=\mu/\beta$,
  using $P^*=(\beta-\mu)/(2\beta)$.  The diagonal and off-diagonal entries are
  \[
    a \;:=\; J_{11}=J_{22} \;=\; \frac{\mu(\alpha-\gamma)-\beta^2}{2\beta},
    \qquad
    b \;:=\; J_{12}=J_{21} \;=\; \frac{2\gamma\mu - \beta(\beta-\mu)}{2\beta}.
  \]
  (The equality $J_{12}=J_{21}$ follows from the symmetry $\alpha_L=\alpha_R$,
  $\gamma_{RL}=\gamma_{LR}$.)  The matrix $\begin{psmallmatrix}a&b\\b&a\end{psmallmatrix}$
  has eigenvalues $\lambda_{1,2}=a\pm b$.  Computing the sum and difference:
  \begin{align*}
    \lambda_1 &= a + b = \frac{\mu(\alpha-\gamma)-\beta^2 + 2\gamma\mu - \beta(\beta-\mu)}{2\beta}
              = \frac{2\mu\beta - 2\beta^2}{2\beta} = \mu - \beta,\\[4pt]
    \lambda_2 &= a - b = \frac{\mu(\alpha-\gamma)-\beta^2 - 2\gamma\mu + \beta(\beta-\mu)}{2\beta}
              = \frac{\mu(\alpha-3\gamma-\beta)}{2\beta}
              = \frac{-4\gamma\mu}{2\beta}
              = -\frac{2\gamma\mu}{\beta}.
  \end{align*}
  Since $\beta>\mu$, both values are strictly negative.  Because $\lambda_1\ne\lambda_2$
  generically, $E_1$ is a stable node.
  $\mathrm{tr}\,J = \lambda_1+\lambda_2 = \mu-\beta-2\gamma\mu/\beta < 0$ and
  $\det J = \lambda_1\lambda_2 = -(\beta-\mu)\cdot 2\gamma\mu/\beta > 0$
  for all $\beta>\mu$, $\gamma,\mu>0$.   
\end{proof}

\begin{remark}[Interpretation of the two eigenvalues]
  The eigenvalue $\lambda_1=\mu-\beta$ is precisely the stability exponent of the
  scalar reduced system (Theorem~\ref{thm:global}): it governs the symmetric mode
  in which $L$ and $R$ change together.  The eigenvalue $\lambda_2=-2\gamma\mu/\beta$
  governs the \emph{asymmetric} mode (the difference $L-R$) and is determined by
  the product of the reactive polarisation rate $\gamma$ and the deradicalisation
  rate $\mu$.  In the limit $\gamma\to0$ (no reactive polarisation), $\lambda_2\to0$:
  the asymmetric mode becomes neutrally stable, reflecting the fact that without
  reactive cross-coupling, the two wings evolve independently.
\end{remark}

\begin{corollary}[No periodic orbits: full asymmetric model]\label{cor:dulac}
  The system \eqref{eq:model} has no closed orbits in the interior of $\Simp$
  for any positive parameter values, without symmetry restrictions.
\end{corollary}

\begin{proof}
  Apply the Bendixson--Dulac criterion with the weight
  $B(L,R)=1/(LRC)>0$ on $\mathrm{int}(\Simp)$, where $C=1-L-R$.
Writing $BF_1 = \alpha_L/R - \mu_L/(RC) + \gamma_{RL}/L$
  and $BF_2 = \alpha_R/L - \mu_R/(LC) + \gamma_{LR}/R$,
  a direct computation gives
  \begin{equation}\label{eq:dulac}
    \frac{\partial(B F_1)}{\partial L}+\frac{\partial(B F_2)}{\partial R}
    = -\frac{\gamma_{RL}}{L^2} - \frac{\gamma_{LR}}{R^2}
      - \frac{\mu_L}{R\,C^2} - \frac{\mu_R}{L\,C^2}.
  \end{equation}
  In $\mathrm{int}(\Simp)$ we have $L,R,C>0$, so all four terms are strictly
  negative, and the sum is strictly negative throughout $\mathrm{int}(\Simp)$.
By Dulac's theorem there are no closed orbits.
  (This result is independent of $\alpha_L,\alpha_R$, and holds for
  all parameter values, including the asymmetric case.  In particular,
  the scalar $S$-equation already precludes periodic orbits in the
  symmetric reduction.)   
\end{proof}

\begin{remark}[$E_0$ stability in the full 2D model]
  Evaluating \eqref{eq:jacobian} at $E_0=(0,0)$ (where $C=1$) gives
  \begin{equation}\label{eq:J0}
    J\big|_{E_0} = \begin{pmatrix}
      \alpha_L - \mu_L & \gamma_{RL} \\[3pt]
      \gamma_{LR}       & \alpha_R - \mu_R
    \end{pmatrix}.
  \end{equation}
  The centrist state $E_0$ is locally asymptotically stable if and only if
  \begin{equation}\label{eq:E0_stab}
    \alpha_L + \alpha_R < \mu_L + \mu_R
    \quad\text{and}\quad
    (\alpha_L-\mu_L)(\alpha_R-\mu_R) > \gamma_{RL}\,\gamma_{LR}.
  \end{equation}
  In the symmetric case, the matrix \eqref{eq:J0} becomes
  $\bigl(\begin{smallmatrix}\alpha-\mu & \gamma\\\gamma & \alpha-\mu\end{smallmatrix}\bigr)$
  with eigenvalues $\lambda_1=\beta-\mu$ and $\lambda_2=\alpha-(\mu+\gamma)$.
  Both are negative if and only if $\beta=\alpha+\gamma<\mu$.
\end{remark}

\begin{theorem}[Global dynamics of the full asymmetric baseline model]
  \label{thm:global_asym}
  Let $\lambda_{\mathrm{PF}}$ be the Perron root from Theorem~\ref{thm:PF}.
  Then:
  \begin{enumerate}
    \item If $\lambda_{\mathrm{PF}}\le 1$, the centrist equilibrium
      $E_0=(0,0)$ is globally asymptotically stable on $\Simp$.
    \item If $\lambda_{\mathrm{PF}}>1$, the unique interior equilibrium
      $E_1$ from Theorem~\ref{thm:PF} is globally asymptotically stable on
      $\Simp\setminus\{E_0\}$.
  \end{enumerate}
\end{theorem}

\begin{proof}
  Write $x=(L,R)^\top$.  The system can be written as
  \[
    \dot x = (K-M)x + N(x),
  \]
  where
  \[
    N(x)=
    \begin{pmatrix}
      -\alpha_L L^2-(\alpha_L+\gamma_{RL})LR-\gamma_{RL}R^2\\[3pt]
      -\gamma_{LR}L^2-(\alpha_R+\gamma_{LR})LR-\alpha_RR^2
    \end{pmatrix}.
  \]
  Thus $N(x)\le0$ componentwise for all $(L,R)\in\Simp$ and $N(x)\neq0$ for
  every $x\neq0$.

  \textbf{Case $\lambda_{\mathrm{PF}}\le1$.}
  Let $q\gg0$ be a positive left Perron eigenvector of the Metzler matrix
  $K-M$, so that
  \[
    q^\top(K-M)=s\,q^\top,
    \qquad s:=s(K-M)\le0.
  \]
  Define $V(x):=q^\top x$.  Then
  \[
    \dot V
    = q^\top(K-M)x + q^\top N(x)
    = s\,V + q^\top N(x)\le0,
  \]
  with strict inequality for every $x\neq0$.  Hence $E_0$ is globally
  asymptotically stable by Lyapunov--LaSalle.

  \textbf{Case $\lambda_{\mathrm{PF}}>1$.}
  Let again $q\gg0$ be the positive left Perron eigenvector of $K-M$, now with
  spectral bound $s=s(K-M)>0$.  Since $N(x)=O(\|x\|^2)$ as $x\to0$, there exists
  a neighbourhood $U$ of $E_0$ in $\mathrm{int}(\Simp)$ such that
  \[
    \dot V \ge \frac{s}{2}V>0
    \qquad\text{for all }x\in U.
  \]
  Hence no trajectory starting in $\mathrm{int}(\Simp)$ can converge to $E_0$.

  By Proposition~\ref{prop:inv}, $\Simp$ is compact and positively invariant.
  By Corollary~\ref{cor:dulac}, there are no periodic orbits in
  $\mathrm{int}(\Simp)$.  The only equilibria are $E_0$ and the unique interior
  equilibrium $E_1$ from Theorem~\ref{thm:PF}.  Therefore, by the
  Poincar\'e--Bendixson theorem \cite{perko2001}, the $\omega$-limit set of any interior
  trajectory must be an equilibrium. Since $E_0$ cannot be an $\omega$-limit
  set of an interior trajectory, one must have
  \[
    \omega(L_0,R_0)=\{E_1\}
    \qquad\text{for every }(L_0,R_0)\in\mathrm{int}(\Simp).
  \]

It remains to show that no trajectory starting in $\mathrm{int}(\Simp)$ can
have an $\omega$-limit point on $\partial\Simp$ other than possibly $E_0$.

On the open edge $\{L=0,\ 0<R<1\}$ one has
\[
  \dot L\big|_{L=0}=\gamma_{RL}R(1-R)>0.
\]
Hence every point of this open edge has a neighbourhood in which the vector
field points strictly into the interior, so no $\omega$-limit set of an
interior trajectory can intersect that edge. The same argument applies to the
open edge $\{R=0,\ 0<L<1\}$.

On the edge $\{C=0\}$, i.e.\ $L+R=1$, one has
\[
  \dot C = -(\dot L+\dot R)=\mu_L L+\mu_R R \ge \min\{\mu_L,\mu_R\}>0.
\]
By continuity, there exists a neighbourhood of $\{C=0\}$ in which $\dot C>0$,
so no interior trajectory can approach this edge in forward time.

Therefore every $\omega$-limit point of an interior trajectory lies in
$\mathrm{int}(\Simp)\cup\{E_0\}$. Since $E_0$ is unstable when
$\lambda_{\mathrm{PF}}>1$, we conclude that
\[
  \omega(L_0,R_0)=\{E_1\}
  \qquad\text{for every }(L_0,R_0)\in\mathrm{int}(\Simp).
\]
Finally, every boundary point of $\Simp\setminus\{E_0\}$ enters the interior
immediately, so the same conclusion holds for all
$(L_0,R_0)\in\Simp\setminus\{E_0\}$.
\end{proof}

\subsection{The bifurcation at \texorpdfstring{$\beta = \mu$}{beta = mu}}

\begin{theorem}[Transcritical bifurcation]\label{thm:bif}
  In the symmetric scalar model \eqref{eq:scalar}, as $\beta=\alpha+\gamma$
  increases through $\mu$:
  \begin{enumerate}
    \item For $\beta<\mu$: only $E_0$ exists in $[0,\tfrac12]$; it is stable.
    \item At $\beta=\mu$: $P^*=0$ coincides with $E_0$; the two equilibria
      collide.
    \item For $\beta>\mu$: $P^*$ bifurcates from $E_0$ into $(0,\tfrac12)$;
      $E_1$ is stable and $E_0$ is unstable.
  \end{enumerate}
  This is a \emph{transcritical bifurcation}: the two branches exchange stability
  as $\beta$ crosses $\mu$.
\end{theorem}

\begin{proof}
  The bifurcation occurs at $(P,\beta)=(0,\mu)$.  Write $f(P,\beta)=P[\beta(1-2P)-\mu]
  =\beta P - 2\beta P^2 - \mu P$.
  At this point: $f=0$, $f_P=\beta-\mu=0$, $f_{PP}=-4\beta\ne0$,
  $f_{P\beta}=1-4P\big|_{P=0}=1\ne0$.
  These are exactly the non-degeneracy conditions for a transcritical bifurcation
  in the classification of \cite{strogatz1994}.   
\end{proof}

\begin{corollary}[Comparative statics]\label{cor:cs}
  At the interior equilibrium $E_1$:
  \begin{enumerate}
    \item $\partial P^*/\partial\alpha > 0$: higher recruitment increases radical
      equilibrium.
    \item $\partial P^*/\partial\gamma > 0$: stronger reactive polarisation
      increases radical equilibrium.
    \item $\partial P^*/\partial\mu < 0$: faster deradicalisation decreases
      radical equilibrium.
    \item $\partial C^*/\partial(\alpha+\gamma) < 0$: the centrist share declines
      with both recruitment and polarisation intensity.
    \item $C^*>0$ for all fixed positive parameters; $C^*\to0$ only in singular
      limits, e.g.\ $\mu\to0$ or $\alpha+\gamma\to\infty$.  In particular,
      centrist collapse is impossible at any fixed positive deradicalisation rate
      and any fixed finite recruitment and polarisation rates.
  \end{enumerate}
\end{corollary}

\begin{proof}
  All statements follow directly from $P^*=\tfrac12(1-\mu/\beta)$ and
  $C^*=\mu/\beta$ with $\beta=\alpha+\gamma$, by differentiation.   
\end{proof}

\section{Numerical Analysis of the Baseline Model}\label{sec:numerics}

All simulations were performed using the \texttt{solve\_ivp} function from
\texttt{SciPy} \cite{scipy2020} with the RK45 solver, relative tolerance
$10^{-10}$, and absolute tolerance $10^{-12}$.  The maximum step size was set
to $0.05$ to ensure accuracy for all parameter configurations.

\subsection{Two qualitative regimes: three representative scenarios}

Figure~\ref{fig:ts} illustrates three representative scenarios within the two
mathematically distinct regimes of Theorem~\ref{thm:global}: the subcritical
case $\beta\le\mu$ and the supercritical coexistence case $\beta>\mu$.  The dashed reference lines
show the theoretical equilibrium values $P^*$ and $C^*=\mu/\beta$ from
Corollary~\ref{cor:Cstar}.

\begin{figure}[ht]
\centering
\includegraphics[width=\textwidth]{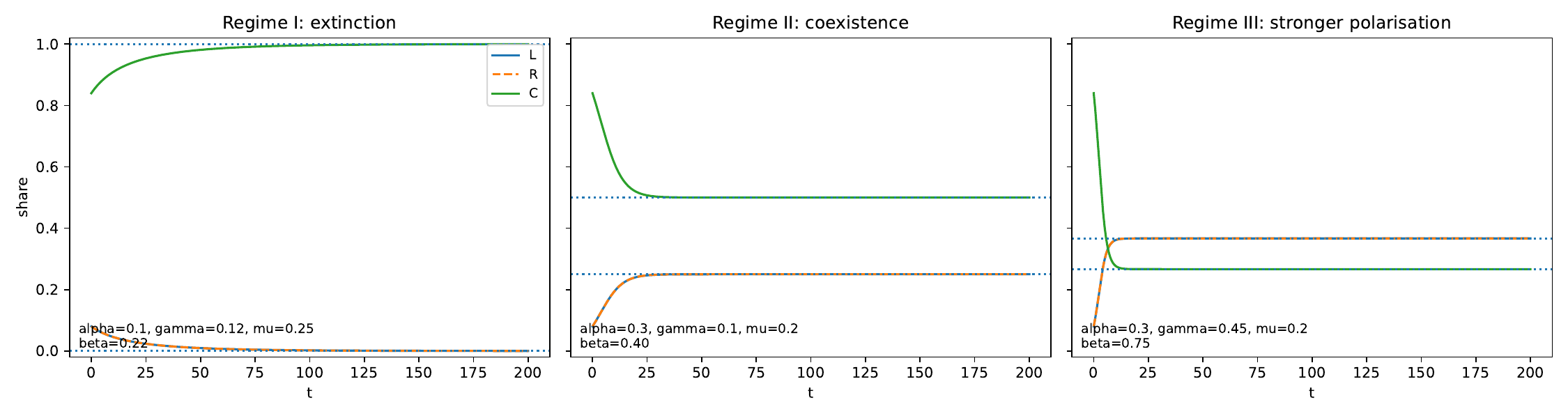}
\caption{Time series for three representative scenarios across the two
  dynamical regimes of Theorem~\ref{thm:global}.
  \textbf{Left} ($\beta=0.22<\mu=0.25$): subcritical case.
    Both radical wings decay; the centrist share approaches~1.
  \textbf{Centre} ($\beta=0.40>\mu=0.20$): supercritical case,
    moderate polarisation.  All groups persist at
    $P^*=0.25$, $C^*=0.50$.
  \textbf{Right} ($\beta=0.75>\mu=0.20$): supercritical case,
    high polarisation.  Strong reactive polarisation shifts the
    equilibrium to $P^*=0.367$, $C^*=0.267$.
  Both right-panel cases belong to the same qualitative regime
  ($\beta>\mu$); they differ only in the degree of polarisation
  at equilibrium.  Dashed lines show analytical predictions;
  numerical and analytical values agree within tolerance $10^{-8}$.}
\label{fig:ts}
\end{figure}

\subsection{Phase portraits}

Figure~\ref{fig:phase} shows phase portraits for the same three parameter sets,
with multiple initial conditions.  The global convergence proved in
Theorem~\ref{thm:global} is confirmed: all trajectories from the interior of
$\Simp$ converge to the same attractor regardless of the starting point.

\begin{figure}[ht]
\centering
\includegraphics[width=\textwidth]{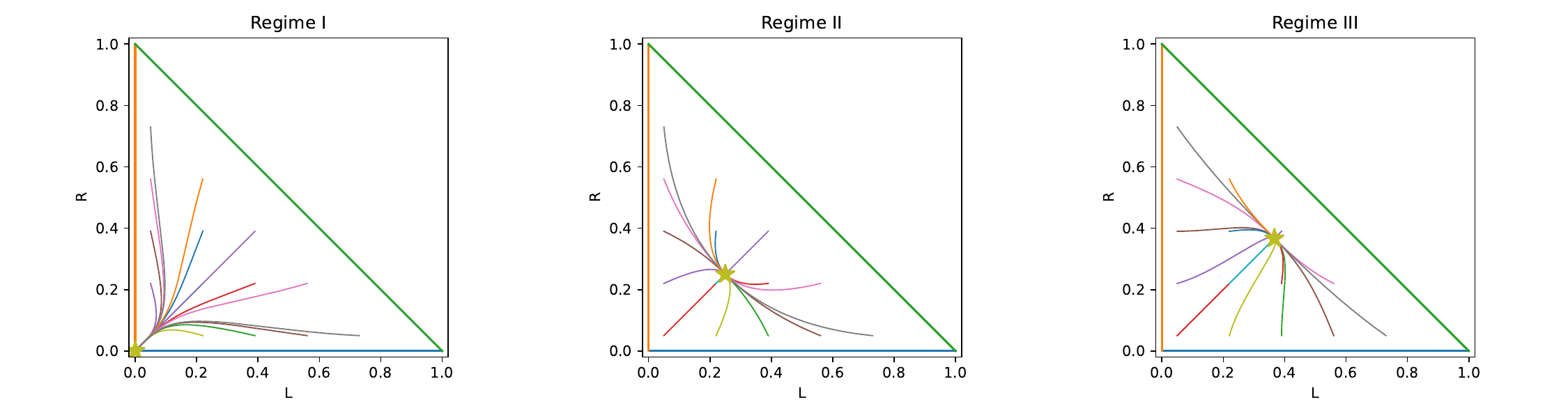}
\caption{Phase portraits in the feasible simplex $\Simp$.  Stars mark attracting
  equilibria.  All trajectories converge to the same fixed point, confirming
  global stability.  In Regime~I the attractor is $E_0=(0,0)$; in Regimes~II
  and~III it is the interior equilibrium $E_1=(P^*,P^*)$.
  No limit cycles are present, consistent with the Dulac criterion
  (Corollary~\ref{cor:dulac}).}
\label{fig:phase}
\end{figure}

\subsection{Bifurcation diagram}

Figure~\ref{fig:bif} shows the dependence of the equilibrium shares $P^*$ and
$C^*$ on the polarisation parameter $\gamma$ for three values of the recruitment
rate $\alpha$.  The transcritical bifurcation at $\beta=\mu$ (i.e.\ $\gamma=\mu-\alpha$)
is clearly visible.  The key qualitative feature, confirmed analytically in
Corollary~\ref{cor:cs}, is that $C^*=\mu/(\alpha+\gamma)>0$ for all finite
parameters: the centrist share is eroded by both recruitment and polarisation but
\emph{never reaches zero}.

\begin{figure}[ht]
\centering
\includegraphics[width=\textwidth]{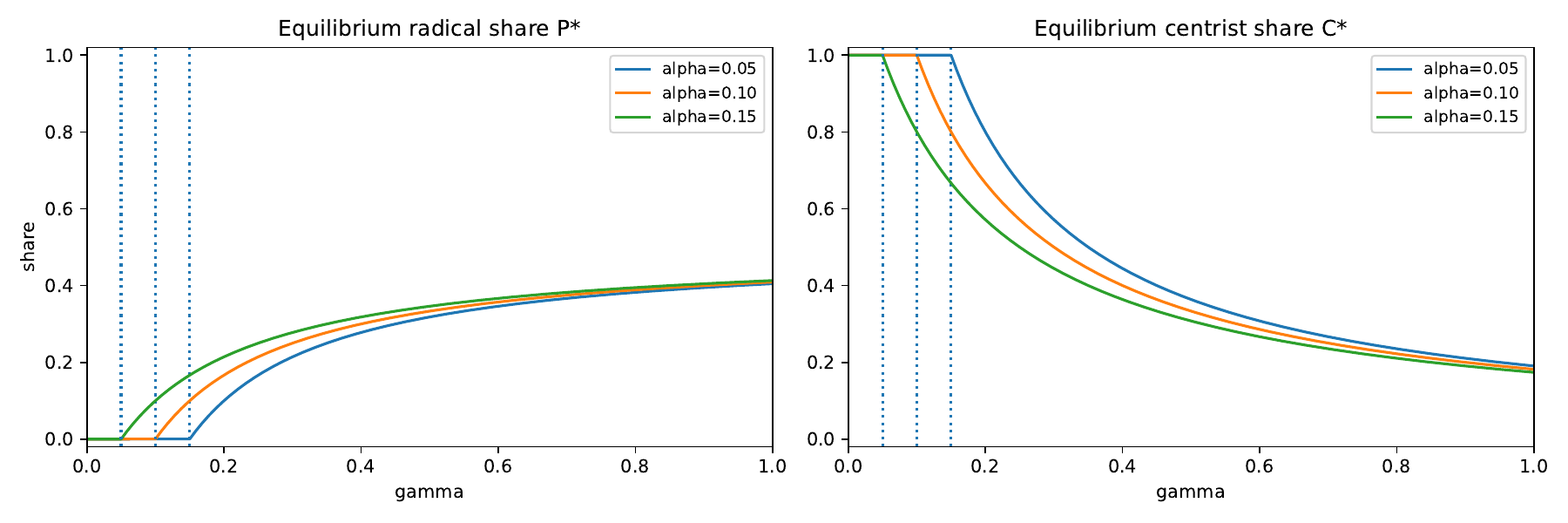}
\caption{Equilibrium radical share $P^*$ (left) and centrist share $C^*$ (right)
  as functions of $\gamma$ for $\alpha\in\{0.05,0.10,0.15\}$ and $\mu=0.20$.
  The bifurcation threshold $\gamma^*=\mu-\alpha>0$ is clearly visible as the
  left endpoint of each curve.
  The hyperbolic decay $C^*=\mu/(\alpha+\gamma)$ (right panel) approaches but
  never reaches zero over the plotted parameter range.}
\label{fig:bif}
\end{figure}

In the asymmetric case the equilibrium shifts toward the stronger wing,
as determined by the Perron eigenvector of $M^{-1}K$ (Theorem~\ref{thm:PF}).
The centrist share at equilibrium remains $C^*=1/\lambda_{\mathrm{PF}}$,
strictly positive for all finite positive parameters.

\paragraph{Additional baseline numerics.}
Further baseline simulations---temporary external shocks of varying intensity
and a sensitivity contour of $C^*=\mu/(\alpha+\gamma)$ across the
$(\alpha,\gamma)$ plane---directly confirm results already proved analytically
in Theorem~\ref{thm:global}, Theorem~\ref{thm:global_asym}, and
Corollary~\ref{cor:Cstar}.  These figures are available as supplementary
material upon request.

\section{The Four-Group Disengagement-Shock Model}\label{sec:model4}

\subsection{Structural limitations of the baseline and motivation}\label{sec:scope}

For fixed positive parameters, the baseline model has a unique global attractor
(Theorem~\ref{thm:global} in the symmetric case,
Theorem~\ref{thm:global_asym} in the asymmetric case).
Consequently, the long-run centrist share depends only on the current parameter
values, and staircase dynamics, shock-history dependence, and bistability are
all impossible.

\begin{corollary}[Impossibility of staircase dynamics]\label{cor:nostaircase}
  Within the baseline model with fixed positive parameters,
  the following are impossible for any initial conditions:
  \begin{enumerate}[label=(\roman*)]
    \item a monotonically decreasing sequence of long-run centrist shares
      $C_1^\infty > C_2^\infty > \cdots$ produced by state perturbations;
    \item a long-run centrist share $C^\infty$ that depends on shock history
      rather than only on the current parameter values;
    \item bistability: two distinct long-run attractors reachable from
      different initial conditions under the same parameters.
  \end{enumerate}
\end{corollary}

\begin{proof}
  Follows directly from uniqueness of the global attractor in
  Theorems~\ref{thm:global} and~\ref{thm:global_asym}.
\end{proof}

Corollary~\ref{cor:nostaircase} is the precise mathematical statement of the
baseline model's structural incompleteness for post-2008 European dynamics.
Two minimal extensions suffice to remove this incompleteness: (a)~a disengaged voter
compartment $A$, so that crisis-induced withdrawal can form a transient
reservoir susceptible to radical mobilisation \cite{norris2019,dalton2004};
and (b)~structural (parameter-shifting) shocks, so that repeated crises can
permanently alter the recruitment/deradicalisation balance---an institutional
channel analysed in a different framework by \cite{levitsky2018}---producing a
cumulative drift $\beta_0 + \sum_k \Delta\beta_k$ that may cross the threshold
$\beta=\mu$.  The following subsections formalise these two additions.

\subsection{State space}

The extension adds a disengaged compartment $A(t)$ to the baseline partition.
Throughout this section, $A(t)$ is interpreted as excess crisis-induced
disengagement above a background abstention level. The conservation law becomes
\begin{equation}\label{eq:conserve4}
  L(t) + R(t) + C(t) + A(t) = 1 \qquad (t\ge0),
\end{equation}
so $C = 1-L-R-A$ and the feasible state space is the standard $3$-simplex
\[
  \mathcal{T} = \{(L,R,A)\in\R^3 : L\ge0,\;R\ge0,\;A\ge0,\;L+R+A\le1\}.
\]

\subsection{Governing equations}

\begin{definition}[Four-group model]\label{def:4group}
  The \emph{four-group disengagement-shock model} is the system
  \begin{equation}\label{eq:4group}
  \begin{aligned}
    \dot{L} &= \alpha_L\,L\,C + \gamma_{RL}\,R\,C + \delta_L\,A\,L - \mu_L\,L, \\
    \dot{R} &= \alpha_R\,R\,C + \gamma_{LR}\,L\,C + \delta_R\,A\,R - \mu_R\,R, \\
    \dot{A} &= \sigma(t)\,C - \delta_L\,A\,L - \delta_R\,A\,R - \rho\,A,
  \end{aligned}
  \end{equation}
  with $C = 1-L-R-A$ and all parameters
  $\alpha_i,\mu_i,\gamma_{ij},\delta_i,\rho > 0$.
\end{definition}

The new terms have direct interpretations.  The factors $\delta_iAX_i$ model
mobilisation of disengaged voters by radical movements; one typically expects
$\delta_i>\alpha_i$, since disengaged voters have already withdrawn from the
mainstream \cite{norris2019}.  The term $\sigma(t)C$ transfers centrists into
disengagement during a crisis, while $\rho A$ returns disengaged voters to the
centre.  The additional parameters are summarised in Table~\ref{tab:params4}.

\begin{table}[ht]
\caption{New parameters introduced in the four-group model.}\label{tab:params4}
\centering
\renewcommand{\arraystretch}{1.4}
\begin{tabular}{lll}
\toprule
Parameter & Interpretation & Expected range \\
\midrule
$\delta_L, \delta_R$ & Apolitical mobilisation rate (by radicals) & $\delta_i > \alpha_i$ \\
$\rho$               & Spontaneous re-engagement rate              & $(0,\infty)$ \\
$\sigma(t)$          & External shock intensity                    & $\ge 0$ \\
\bottomrule
\end{tabular}
\end{table}

\subsection{The impulse shock model}\label{sec:impulse}

A crisis is modelled as an instantaneous transfer of a fraction of the current
centrist pool into disengagement.

\begin{definition}[Impulse shock]\label{def:impulse}
  An \emph{impulse shock} of amplitude $s>0$ at time $t_0$ corresponds to
  setting $\sigma(t) = s\,\delta_{\mathrm{D}}(t-t_0)$ in \eqref{eq:4group},
  where $\delta_{\mathrm{D}}$ is the Dirac delta distribution.  This produces
  the jump condition at $t_0$:
  \begin{equation}\label{eq:jump}
    C(t_0^+) = (1-\Delta)\,C(t_0^-), \qquad
    A(t_0^+) = A(t_0^-) + \Delta\,C(t_0^-),
  \end{equation}
  with $\Delta := 1-e^{-s}\approx s$ for small $s$, and $L$, $R$ unchanged.
  The system then evolves autonomously from $(L_0,R_0,A_0)$ with $\sigma\equiv0$.
\end{definition}

After the jump the post-shock dynamics are autonomous, which cleanly separates
the shock event from its subsequent consequences.

\subsection{Structural vs.\ state shocks}\label{sec:structural}

Definition~\ref{def:impulse} describes a \emph{state shock}: the parameters
remain unchanged and only the state is perturbed.  We additionally define:

\begin{definition}[Structural shock]\label{def:structural}
  Let $\theta$ denote the vector of model parameters. A
  \emph{structural shock} at time $t_0$ consists of:
  \begin{enumerate}[label=(\roman*)]
    \item a state component: an impulse shock of amplitude $\Delta$ as in
      Definition~\ref{def:impulse};
    \item a parameter component: a permanent parameter change
      $\theta^- \to \theta^+$ taking effect at $t_0$ and persisting for all
      $t>t_0$.
  \end{enumerate}
  A shock with $\theta^+=\theta^-$ is a \emph{pure state shock}. In the
  symmetric reduction considered below, the parameter component is taken to be
  the scalar shift $\beta^+=\beta^-+\Delta\beta$ with the remaining parameters
  unchanged.
\end{definition}

A pure state shock changes only the point in state space; a structural shock
changes both the state and the post-shock parameter regime.  This distinction
is the organising principle of the extended model.

\section{Mathematical Analysis of the Four-Group Model}\label{sec:analysis4}

\subsection{Forward invariance}

\begin{proposition}[Forward invariance of $\mathcal{T}$]\label{prop:inv4}
  For any $(L_0,R_0,A_0)\in\mathcal{T}$ and positive parameters, the
  solution of \eqref{eq:4group} with $\sigma\equiv0$ remains in
  $\mathcal{T}$ for all $t\ge0$.
\end{proposition}

\begin{proof}
  We verify that the vector field points inward (or is tangent) on each
  face of $\mathcal{T}$.

  \emph{Face $L=0$:}
  $\dot{L}\big|_{L=0} = \gamma_{RL}\,R\,C \ge 0$.

  \emph{Face $R=0$:}
  $\dot{R}\big|_{R=0} = \gamma_{LR}\,L\,C \ge 0$.

  \emph{Face $A=0$:}
  $\dot{A}\big|_{A=0} = \sigma(t)\,C \ge 0$
  (with $\sigma \equiv 0$ post-shock, $\dot{A}\big|_{A=0} = 0$).

  \emph{Face $C=0$ (i.e.\ $L+R+A=1$):}
  Here $C=0$.  Then $\dot{L}\big|_{C=0} = \delta_L A L - \mu_L L = L(\delta_L A - \mu_L)$
  and similarly for $R$.  Also $\dot{A}\big|_{C=0} = -A(\delta_L L + \delta_R R + \rho) \le 0$.
  Therefore
\begin{align*}
    \dot{L}+\dot{R}+\dot{A}\big|_{C=0}
    &= L(\delta_L A-\mu_L) + R(\delta_R A-\mu_R) - A(\delta_L L + \delta_R R + \rho)\\
    &= -\mu_L L - \mu_R R - \rho A \le 0.
  \end{align*}
  Hence $\dot{C} = -(\dot{L}+\dot{R}+\dot{A}) \ge 0$: the total $L+R+A$
  cannot increase beyond~1.

  By the Picard--Lindel\"of theorem, the solution is unique, and $\mathcal{T}$
  is positively invariant.   
\end{proof}

\begin{remark}
  Unlike the baseline, the face $C=0$ now admits trajectories that dwell on
  it for finite time if $\mu_L = \mu_R = \rho = 0$.  For all strictly
  positive parameters (our standing assumption), the system is forced back
  into the interior immediately, since
  $\dot{C}\big|_{C=0} = \mu_L L + \mu_R R + \rho A > 0$
  whenever $(L,R,A)\ne(0,0,0)$.
\end{remark}

\subsection{Symmetric reduction}

\begin{definition}[Symmetric four-group model]\label{def:sym4}
  The model \eqref{eq:4group} is \emph{symmetric} if
  $\alpha_L=\alpha_R=\alpha$, $\mu_L=\mu_R=\mu$, $\gamma_{LR}=\gamma_{RL}=\gamma$,
  $\delta_L=\delta_R=\delta$, and $L(0)=R(0)$.
\end{definition}

Under the symmetric ansatz, $L(t)=R(t)=:P(t)$ for all $t$ (by the same
diagonal-invariance argument as in the baseline), and the system reduces to
the \emph{two-dimensional post-shock system}:
\begin{equation}\label{eq:2d}
\begin{aligned}
  \dot{P} &= P\bigl[(\alpha+\gamma)(1-2P-A) + \delta A - \mu\bigr]
           = P\bigl[(\beta-\mu) - 2\beta P + (\delta-\beta)A\bigr], \\
  \dot{A} &= -(2\delta P + \rho)\,A,
\end{aligned}
\end{equation}
where $\beta:=\alpha+\gamma$ and $C=1-2P-A$.  The feasible region for
\eqref{eq:2d} is
\[
  \mathcal{F} = \{(P,A) : P\ge0,\;A\ge0,\;2P+A\le1\}.
\]

\subsection{Equilibrium classification}

\begin{theorem}[Equilibria of the symmetric four-group model]\label{thm:eq4}
  System \eqref{eq:2d} with $\sigma\equiv0$ has exactly two equilibria
  in $\mathcal{F}$:
  \begin{enumerate}[label=(\roman*)]
    \item $E_0 = (P=0,\,A=0)$: the centrist-only state ($C=1$).
    \item $E_1 = \bigl(P^*=\tfrac12(1-\mu/\beta),\;A=0\bigr)$:
      coexistence, existing in the interior of $\mathcal{F}$ if and only
      if $\beta>\mu$.
  \end{enumerate}
  No equilibrium with $A^*>0$ exists.
\end{theorem}

\begin{proof}
  From $\dot{A}=0$: since $2\delta P+\rho>0$, we obtain $A^*=0$.
  With $A=0$, $\dot{P}=P[(\beta-\mu)-2\beta P]=0$ gives $P=0$ or
  $P^*=\tfrac12(1-\mu/\beta)$.  The condition $P^*\in(0,\tfrac12)$
  is equivalent to $\beta>\mu$.   
\end{proof}

\begin{remark}[Interpretation of $A(t)\to0$ and transience of disengagement]
\label{rem:Abase}
  Theorem~\ref{thm:eq4} establishes that $A^*=0$ at every equilibrium,
  which might appear to predict 100\% voter turnout in the long run.
  The correct interpretation is that $A(t)$ models the \emph{excess}
  (crisis-induced) disengagement above a stable baseline $A_{\mathrm{base}}>0$
  representing chronic non-participation; $\tilde{A}(t)\to0$ means the system
  returns to its pre-crisis abstention level, not that abstention disappears.
  The transience of $A$ is not a modelling assumption but a consequence of
  the positive depletion rates ($\rho>0$, $\delta>0$): in the long run every
  voter either joins a radical movement or returns to the mainstream.  The
  empirically important dynamics occur in the transient phase, which can last
  for years if $\rho$ is small.
\end{remark}

\begin{proposition}[Equilibrium inheritance from the baseline]\label{prop:inherit4}
  For the full four-group model \eqref{eq:4group} with $\sigma\equiv0$ and
  all parameters positive, every equilibrium has $A^*=0$. Consequently,
  the equilibria of the full four-group model are exactly the equilibria
  of the baseline model \eqref{eq:model} on the invariant face $A=0$.

  Moreover, at any such equilibrium the Jacobian is block upper-triangular:
  \[
    J_4(L^*,R^*,0)=
    \begin{pmatrix}
      J_{\mathrm{base}}(L^*,R^*) & *\\
      0\;\;0 & -(\delta_L L^*+\delta_R R^*+\rho)
    \end{pmatrix},
  \]
  so local stability is inherited from the baseline model together with one
  additional strictly negative transverse eigenvalue
  $-(\delta_L L^*+\delta_R R^*+\rho)<0$.
\end{proposition}

\begin{proof}
  At any equilibrium, $0=\dot{A}=-A(\delta_L L+\delta_R R+\rho)$.  Since
  $\delta_L,\delta_R,\rho>0$, the coefficient is strictly positive, hence
  $A^*=0$.  Substituting $A=0$ reduces \eqref{eq:4group} to exactly the
  baseline equations~\eqref{eq:model}.  The block-triangular Jacobian
  follows by direct differentiation.   
\end{proof}

\begin{proposition}[Exponential decay of the disengaged compartment]
  \label{prop:Adecay4}
  For every solution of the post-shock autonomous system \eqref{eq:4group}
  with $\sigma\equiv0$,
  \[
    A(t)=A(0)\exp\!\left(
      -\int_0^t (\delta_L L(s)+\delta_R R(s)+\rho)\,ds
    \right)
    \le A(0)e^{-\rho t}.
  \]
  In particular, $A(t)\to0$ exponentially as $t\to\infty$.
\end{proposition}

\begin{proof}
  The equation
  \[
    \dot A = -(\delta_L L+\delta_R R+\rho)A
  \]
  is scalar and separable.  Integrating from $0$ to $t$ gives the exact
  formula.  Since $\delta_L L+\delta_R R+\rho\ge\rho>0$, the exponential bound
  follows immediately.   
\end{proof}

\begin{theorem}[Global dynamics of the full asymmetric four-group model]
  \label{thm:global4}
  Consider the post-shock autonomous system \eqref{eq:4group} with
  $\sigma\equiv0$.  Define
  \[
    K:=\begin{pmatrix}
      \alpha_L & \gamma_{RL}\\
      \gamma_{LR} & \alpha_R
    \end{pmatrix},
    \qquad
    M:=\begin{pmatrix}
      \mu_L & 0\\
      0 & \mu_R
    \end{pmatrix},
  \]
  and let
  \[
    \mathcal R_{\mathrm{rad}}:=\lambda_{\mathrm{PF}}(M^{-1}K).
  \]
  Then:
  \begin{enumerate}[label=(\roman*)]
    \item If $\mathcal R_{\mathrm{rad}}\le1$, the centrist-only equilibrium
      $E_0=(L,R,A)=(0,0,0)$ is globally asymptotically stable on $\mathcal T$.
    \item If $\mathcal R_{\mathrm{rad}}>1$, let
      $E_1=(L^*,R^*,0)$ be the unique radicalised equilibrium from
      Theorem~\ref{thm:PF4}(ii).  The radical-free axis
      \[
        \Gamma:=\{(L,R,A)\in\mathcal T:\;L=R=0\}
      \]
      is forward invariant, and every trajectory on $\Gamma$ converges to
      $E_0$.  For every initial condition in $\mathcal T\setminus\Gamma$,
      one has
      \[
        (L(t),R(t),A(t))\to E_1
        \qquad (t\to\infty).
      \]
  \end{enumerate}
\end{theorem}

\begin{proof}
  Write $x=(L,R)^\top$ and
  \[
    D:=\begin{pmatrix}\delta_L&0\\0&\delta_R\end{pmatrix}.
  \]
  Using $C=1-L-R-A$, the $(L,R)$-subsystem can be written exactly as
  \[
    \dot x = \bigl[(1-A)K + AD - M\bigr]x - (L+R)Kx
           = (K-M)x + A(D-K)x - (L+R)Kx.
  \]

  \textbf{Part (i).}
  Let $q\gg0$ be a positive left Perron eigenvector of the irreducible Metzler
  matrix $K-M$, normalised so that
  \[
    q^\top(K-M)=s_0 q^\top,
    \qquad s_0:=s(K-M)\le0.
  \]
  Define
  \[
    V(x):=q^\top x.
  \]
  Choose
  \[
    \eta>\max\!\left\{
      0,\,
      \frac{(q^\top(D-K))_1}{\delta_L},\,
      \frac{(q^\top(D-K))_2}{\delta_R}
    \right\},
  \]
  and set
  \[
    W(L,R,A):=V(x)+\eta A.
  \]
  Differentiating along trajectories gives
  \begin{align*}
    \dot W
    &= s_0V + A\,q^\top(D-K)x - (L+R)\,q^\top Kx
       - \eta(\delta_L L+\delta_R R+\rho)A \\
    &= s_0V - (L+R)\,q^\top Kx - \eta\rho A \\
    &\qquad
       + A\Bigl(\bigl[(q^\top(D-K))_1-\eta\delta_L\bigr]L
       + \bigl[(q^\top(D-K))_2-\eta\delta_R\bigr]R\Bigr).
  \end{align*}
  By construction of $\eta$, the last line is nonpositive.  Since
  $s_0\le0$, $q\gg0$, and $K$ has strictly positive entries, we obtain
  \[
    \dot W<0
    \qquad\text{for every }(L,R,A)\neq(0,0,0).
  \]
  Thus $W$ is a strict Lyapunov function on the compact positively invariant
  set $\mathcal T$, and $E_0$ is globally asymptotically stable.

  \textbf{Part (ii).}
  On the axis $\Gamma$ one has $L=R=0$ and therefore
  \[
    \dot A=-\rho A,
  \]
  so $\Gamma$ is forward invariant and every trajectory on it converges to
  $E_0$.

  Now let the initial condition lie in $\mathcal T\setminus\Gamma$.  By
  uniqueness of solutions and invariance of $\Gamma$, one has
  $x(t)\neq0$ for all $t\ge0$, hence $V(t)>0$ for all $t\ge0$.
  Since $\mathcal R_{\mathrm{rad}}>1$, the spectral bound
  \[
    s_0:=s(K-M)>0.
  \]
  By Proposition~\ref{prop:Adecay4}, $A(t)\to0$ exponentially.

  Define
  \[
    \kappa:=\max_{i=1,2}\frac{|(q^\top(D-K))_i|}{q_i},
    \qquad
    c_K:=\max_{i=1,2}\frac{(q^\top K)_i}{q_i},
    \qquad
    q_{\min}:=\min\{q_1,q_2\}.
  \]
  Then
  \[
    q^\top(D-K)x \ge -\kappa\,q^\top x = -\kappa V,
    \qquad
    q^\top Kx \le c_K\,q^\top x = c_KV,
    \qquad
    L+R \le \frac{V}{q_{\min}}.
  \]
  Substituting into the exact identity for $\dot V$ yields
  \[
    \dot V
    \ge \left[s_0-\kappa A(t)-\frac{c_K}{q_{\min}}V\right]V.
  \]
  Choose $T>0$ such that $\kappa A(t)\le s_0/4$ for all $t\ge T$, and set
  \[
    \varepsilon:=\frac{s_0q_{\min}}{4c_K}>0.
  \]
  Then for all $t\ge T$ and all $0<V(t)\le\varepsilon$,
  \[
    \dot V(t)\ge \frac{s_0}{2}V(t)>0.
  \]
  It follows that there exists $T_1\ge T$ such that
  \[
    V(t)\ge\varepsilon
    \qquad\text{for all }t\ge T_1.
  \]
  Indeed, if $V(t)\le\varepsilon$ for all sufficiently large $t$, the above
  differential inequality would force exponential growth, while any crossing
  from $V=\varepsilon$ to $V<\varepsilon$ is impossible because $\dot V>0$ on
  the boundary $V=\varepsilon$ for $t\ge T$.

  Therefore every $\omega$-limit set of a trajectory starting in
  $\mathcal T\setminus\Gamma$ is contained in
  \[
    \Sigma_2^\varepsilon
    :=\{(L,R,0)\in\mathcal T:\;q^\top(L,R)^\top\ge\varepsilon\}
    \subset \Sigma_2\setminus\{E_0\},
  \]
  where
  \[
    \Sigma_2:=\{(L,R,0)\in\mathcal T\}
  \]
  is the baseline face. By Proposition~\ref{prop:Adecay4}, one has $A(t)\to0$ exponentially, so every
$\omega$-limit point satisfies $A=0$. Since the full post-shock system is
autonomous and its restriction to the invariant face $A=0$ coincides exactly
with the baseline flow \eqref{eq:model}, every $\omega$-limit set of a
trajectory in $\mathcal T\setminus\Gamma$ is a compact invariant subset of the
baseline face $\Sigma_2$. Because it is contained in
$\Sigma_2^\varepsilon\subset\Sigma_2\setminus\{E_0\}$, Theorem~\ref{thm:global_asym}
implies that this $\omega$-limit set must be $\{E_1\}$.

  By Theorem~\ref{thm:global_asym}, the equilibrium $E_1$ is globally
  asymptotically stable for the baseline flow on
  $\Sigma_2\setminus\{E_0\}$.  Since every $\omega$-limit set is a compact
  invariant subset of $\Sigma_2^\varepsilon$, the attraction to $E_1$ is
  uniform on that set by continuity and compactness. Hence, for every
  neighbourhood $U$ of $E_1$, there exists $T_U$ such that the baseline flow
  $\phi_t$ satisfies
  \[
    \phi_t(p)\in U
    \qquad\text{for all }p\in\omega(L_0,R_0,A_0),\; t\ge T_U.
  \]
  Because the $\omega$-limit set is invariant under $\phi_t$, we have
  \[
    \omega(L_0,R_0,A_0)=\phi_t\bigl(\omega(L_0,R_0,A_0)\bigr)\subset U
    \qquad (t\ge T_U).
  \]
  Since $U$ is arbitrary,
  \[
    \omega(L_0,R_0,A_0)=\{E_1\},
  \]
  and therefore $(L(t),R(t),A(t))\to E_1$.
\end{proof}

\begin{remark}[Relation to next-generation thresholds]\label{rem:ngm}
  The quantity
  \[
    \mathcal R_{\mathrm{rad}}
    := \lambda_{\mathrm{PF}}(M^{-1}K)
  \]
is the political analogue of a basic reproduction number.  In epidemiological 
  compartment models, thresholds of exactly this type arise from the 
  next-generation matrix framework of \cite{diekmann1990} and 
  \cite{vandendriessche2002}: the matrix $M^{-1}K$ plays the role of the 
  next-generation operator, $K$ collecting new-infection (here: recruitment and 
  reactive-polarisation) terms and $M$ collecting removal (here: deradicalisation) 
  terms.  Here the two radical wings play the role of ``infectious'' compartments. 
  Thus the threshold $\mathcal{R}_{\mathrm{rad}}=1$ has the same mathematical 
  status as $R_0=1$: below it, the centrist equilibrium is stable; above it, 
  radical persistence becomes possible.
\end{remark}

\begin{theorem}[Perron--Frobenius structural threshold in the full asymmetric
  four-group model]\label{thm:PF4}
  Consider the post-shock autonomous system \eqref{eq:4group} with
  $\sigma\equiv0$ and post-shock parameter vector $\theta^+$.  Define
  \[
    K^+ :=
    \begin{pmatrix}
      \alpha_L^+ & \gamma_{RL}^+\\
      \gamma_{LR}^+ & \alpha_R^+
    \end{pmatrix},
    \qquad
    M^+ :=
    \begin{pmatrix}
      \mu_L^+ & 0\\
      0 & \mu_R^+
    \end{pmatrix},
    \qquad
    B^+ := (M^+)^{-1}K^+,
  \]
  and let
  \[
    \mathcal R_{\mathrm{rad}}^+ := \lambda_{\mathrm{PF}}(B^+).
  \]
  Then:
  \begin{enumerate}[label=(\roman*)]
    \item The centrist-only equilibrium $E_0=(L,R,A)=(0,0,0)$ is locally
      asymptotically stable if and only if $\mathcal R_{\mathrm{rad}}^+<1$,
      nonhyperbolic at $\mathcal R_{\mathrm{rad}}^+=1$, and unstable if
      $\mathcal R_{\mathrm{rad}}^+>1$.
    \item A nontrivial equilibrium with $L^*>0$ and $R^*>0$ exists if and only if
      $\mathcal R_{\mathrm{rad}}^+>1$.  When it exists, it is unique and has the form
      \[
        (L^*,R^*,A^*)=(L^*,R^*,0),
        \qquad
        C^*=\frac{1}{\mathcal R_{\mathrm{rad}}^+},
      \]
      where $(L^*,R^*)$ is determined by the positive Perron eigenvector of $B^+$.
      Its local stability is inherited from the baseline equilibrium of
      Theorem~\ref{thm:PF} via Proposition~\ref{prop:inherit4}.
    \item In the symmetric reduction
      \[
        \alpha_L^+=\alpha_R^+=\alpha^+,\qquad
        \mu_L^+=\mu_R^+=\mu^+,\qquad
        \gamma_{LR}^+=\gamma_{RL}^+=\gamma^+,
      \]
      one has
      \[
        \mathcal R_{\mathrm{rad}}^+
        = \frac{\alpha^++\gamma^+}{\mu^+}
        = \frac{\beta^+}{\mu^+},
      \]
      so the full threshold $\mathcal R_{\mathrm{rad}}^+=1$ reduces exactly to
      the scalar transcritical threshold $\beta^+=\mu^+$.
  \end{enumerate}
\end{theorem}

\begin{proof}
  At $E_0=(0,0,0)$ we have $C=1$, $A=0$, and the Jacobian of \eqref{eq:4group}
  with respect to $(L,R,A)$ is
  \[
    J_4(E_0)=
    \begin{pmatrix}
      \alpha_L^+-\mu_L^+ & \gamma_{RL}^+ & 0\\
      \gamma_{LR}^+ & \alpha_R^+-\mu_R^+ & 0\\
      0 & 0 & -\rho^+
    \end{pmatrix}
    =
    \begin{pmatrix}
      K^+-M^+ & 0\\
      0 & -\rho^+
    \end{pmatrix}.
  \]
  Hence local stability of $E_0$ is equivalent to the sign of the spectral
  bound of the Metzler matrix $K^+-M^+$.  By the standard M-matrix criterion,
  \[
    s(K^+-M^+)<0
    \iff
    \lambda_{\mathrm{PF}}\!\big((M^+)^{-1}K^+\big)<1.
  \]
Likewise,
  \begin{align*}
    s(K^+-M^+)=0 &\iff \lambda_{\mathrm{PF}}\!\big((M^+)^{-1}K^+\big)=1,\\[2pt]
    s(K^+-M^+)>0 &\iff \lambda_{\mathrm{PF}}\!\big((M^+)^{-1}K^+\big)>1.
  \end{align*}
  This proves part~(i).

  Part~(ii) follows from Proposition~\ref{prop:inherit4}: every equilibrium
  of the full four-group model has $A^*=0$, so equilibria are exactly the
  baseline equilibria for the post-shock parameter set $\theta^+$.  Applying
  Theorem~\ref{thm:PF} to that parameter set yields existence iff
  $\mathcal R_{\mathrm{rad}}^+>1$, uniqueness, the formula
  $C^*=1/\mathcal R_{\mathrm{rad}}^+$, and local stability.

  Part~(iii) is immediate, since in the symmetric case
  \[
    B^+ = \frac{1}{\mu^+}
    \begin{pmatrix}
      \alpha^+ & \gamma^+\\
      \gamma^+ & \alpha^+
    \end{pmatrix}
  \]
  has Perron root $(\alpha^++\gamma^+)/\mu^+ = \beta^+/\mu^+$.   
\end{proof}

\begin{remark}[Relation to the symmetric shock theorems]\label{rem:PF4bridge}
  Theorems~\ref{thm:critical}--\ref{thm:cumulative} are stated on the symmetric
  invariant manifold $L=R$, where the transient post-shock dynamics admit
  closed formulas.  Theorem~\ref{thm:PF4} shows that the scalar threshold
  $\beta=\mu$ used there is not ad hoc: it is precisely the symmetric special
  case of the full asymmetric Perron--Frobenius threshold
  $\mathcal R_{\mathrm{rad}}^+=1$.

  In particular, the location of the structural bifurcation is determined by
  the balance between recruitment/reactive-polarisation and deradicalisation
  parameters $(\alpha_i,\gamma_{ij},\mu_i)$, while the mobilisation and
  re-engagement parameters $(\delta_i,\rho)$ govern transient post-shock
  amplification and decay.
\end{remark}

\subsection{Stability analysis}

\begin{proposition}[Stability of equilibria in the symmetric model]\label{prop:stab4}
  For system \eqref{eq:2d} with $\sigma\equiv0$:
  \begin{enumerate}[label=(\roman*)]
    \item If $\beta\le\mu$: $E_0=(0,0)$ is \emph{globally asymptotically stable}
      on $\mathcal{F}$; $P(t)\to0$ and $A(t)\to0$ for all initial conditions.
    \item If $\beta>\mu$: the edge $\{P=0\}$ is forward-invariant.
      \begin{itemize}
        \item Every initial condition with $P_0=0$ satisfies $P(t)\equiv0$
          and $A(t)\to0$; the trajectory converges to $E_0$.
        \item Every initial condition with $P_0>0$ converges to
          $E_1=\bigl(\tfrac12(1-\mu/\beta),\,0\bigr)$.
      \end{itemize}
      In this regime $E_0$ is a saddle and $E_1$ is locally asymptotically stable.
  \end{enumerate}
\end{proposition}

\begin{proof}
  \textbf{Case (i).}  Define the Lyapunov function $V(P,A):=P+A \ge 0$.  Using
  \eqref{eq:2d},
  \[
    \dot{V}
    = P\bigl[(\beta-\mu)-2\beta P+(\delta-\beta)A\bigr]-(2\delta P+\rho)A
    = (\beta-\mu)P - 2\beta P^2 - (\beta+\delta)PA - \rho A.
  \]
  When $\beta\le\mu$ every term is non-positive, so $\dot{V}\le0$, with
  $\dot{V}=0$ only at $(P,A)=(0,0)$.  Global asymptotic stability of $E_0$
  follows from LaSalle's invariance principle.

  \textbf{Case (ii).}  If $P_0=0$, then $\dot{P}\big|_{P=0}=0$, so $P(t)\equiv0$
  and the $\dot{A}$ equation reduces to $\dot{A}=-\rho A$, giving
  $A(t)=A_0e^{-\rho t}\to0$.

  Now suppose $P_0>0$.  Positivity of the $P$-equation implies $P(t)>0$
  for all $t\ge0$.  Since
  \[
    \dot A = -(2\delta P+\rho)A,
  \]
  we have $A(t)\to0$ exponentially.

  Fix any $\eta\in(0,\beta-\mu)$.  Because $A(t)\to0$, there exists
  $T_\eta>0$ such that
  \[
    |(\delta-\beta)A(t)|\le \eta
    \qquad\text{for all }t\ge T_\eta.
  \]
  Therefore, for all $t\ge T_\eta$,
  \[
    P\bigl[(\beta-\mu)-\eta-2\beta P\bigr]
    \;\le\;
    \dot P
    \;\le\;
    P\bigl[(\beta-\mu)+\eta-2\beta P\bigr].
  \]

  Let $z_-(t)$ and $z_+(t)$ solve the logistic comparison equations
  \[
    \dot z_\pm
    = z_\pm\bigl[(\beta-\mu)\pm\eta-2\beta z_\pm\bigr],
    \qquad
    z_\pm(T_\eta)=P(T_\eta).
  \]
  By scalar comparison,
  \[
    z_-(t)\le P(t)\le z_+(t)
    \qquad\text{for all }t\ge T_\eta.
  \]
  Since each logistic equation has a globally attracting positive equilibrium,
  \[
    z_\pm(t)\to \frac{\beta-\mu\pm\eta}{2\beta}
    \qquad (t\to\infty),
  \]
  we obtain
  \[
    \frac{\beta-\mu-\eta}{2\beta}
    \le
    \liminf_{t\to\infty}P(t)
    \le
    \limsup_{t\to\infty}P(t)
    \le
    \frac{\beta-\mu+\eta}{2\beta}.
  \]
  Letting $\eta\to0$ yields
  \[
    P(t)\to P^*=\frac12\left(1-\frac{\mu}{\beta}\right).
  \]
  Since also $A(t)\to0$, it follows that $(P(t),A(t))\to E_1$.

  The Jacobians confirm the type of each equilibrium:
  \[
    J(E_0)=\begin{pmatrix}\beta-\mu & 0\\ 0 & -\rho\end{pmatrix},\qquad
    J(E_1)=\begin{pmatrix}-(\beta-\mu) & P^*(\delta-\beta)\\0 & -(2\delta P^*+\rho)
    \end{pmatrix}.
  \]
  At $E_0$: eigenvalues $\beta-\mu>0$ and $-\rho<0$ (saddle).
  At $E_1$: both eigenvalues are real and negative, because $J(E_1)$ is
upper triangular. Hence $E_1$ is a stable node
(possibly degenerate/improper if the diagonal entries coincide).   
\end{proof}

\begin{remark}[The radical-seed requirement]\label{rem:seed}
  The invariance of $\{P=0\}$ means that the model cannot generate radicals
  from a completely zero starting point.  In the symmetric reduced system,
  if $P_0=0$ then $P(t)\equiv0$ regardless of the size of $A$ or any
  structural parameter shift $\Delta\beta$.  This is not a defect of the
  model but a reflection of political reality: a radical movement must have
  some pre-existing organisational core to exploit the disengaged pool.
  All results in Sections~\ref{sec:cumulative} and~\ref{sec:empirical}
  implicitly assume $P_0>0$; in Europe this assumption is almost universally
  satisfied.
\end{remark}

\subsection{The critical shock amplitude}

We now quantify the conditions under which an impulse shock produces a
radicalization surge.  Suppose the pre-shock state is near $E_0$: $P_0=\varepsilon\ll1$,
$A_0=\Delta$ (from the jump condition \eqref{eq:jump}).

\begin{theorem}[Critical shock amplitude]\label{thm:critical}
  Suppose $\beta<\mu$ (Regime~I: democracy is stable) and $\delta>\beta$
  (apolitical mobilisation more efficient than direct centrist recruitment).
  Immediately after the impulse shock, the radical share $P$ experiences
  an initial surge ($\dot{P}>0$) if and only if
  \begin{equation}\label{eq:Delta_c}
    \Delta > \Delta_c := \frac{\mu-\beta}{\delta-\beta}.
  \end{equation}
\end{theorem}

\begin{proof}
  Evaluating $\dot{P}/P$ from \eqref{eq:2d} at $(P,A)=(\varepsilon,\Delta)$
  and taking $\varepsilon\to0$:
  \[
    \frac{\dot{P}}{P}\bigg|_{\varepsilon\to0}
    = (\beta-\mu) + (\delta-\beta)\Delta.
  \]
  This is positive iff $(\delta-\beta)\Delta > \mu-\beta$,
  i.e.\ $\Delta > (\mu-\beta)/(\delta-\beta) = \Delta_c$.   
\end{proof}

\begin{remark}[Exact instantaneous growth criterion]\label{rem:Ac}
  For a general post-shock symmetric state $(P_0,A_0)$ with $\delta>\beta$,
  the exact condition for an initial surge is
  \[
    \dot{P}(0)>0
    \iff
    A_0 > A_c(P_0) := \frac{\mu-\beta+2\beta P_0}{\delta-\beta}.
  \]
  Theorem~\ref{thm:critical} is the near-centrist approximation $P_0=\varepsilon\ll1$,
  giving $A_c\approx\Delta_c=(\mu-\beta)/(\delta-\beta)$.  For a system
  already substantially radicalised ($P_0$ not small), the effective threshold
  $A_c(P_0)>\Delta_c$ is higher: pre-existing radical activity partially
  consumes the disengaged pool before it can fuel further growth.
\end{remark}

\begin{corollary}[Properties of $\Delta_c$]\label{cor:Dc}
  The critical threshold $\Delta_c = (\mu-\beta)/(\delta-\beta)$ satisfies:
  \begin{itemize}
    \item $\Delta_c > 0$ whenever $\mu > \beta$ and $\delta > \beta$.
    \item $\Delta_c < 1$ if and only if $\delta > \mu$.  When $\beta < \delta \le \mu$,
      we have $\Delta_c \ge 1$, meaning no feasible shock ($\Delta\le1$) can
      trigger a surge---the system is shock-proof regardless of crisis magnitude.
      When $\delta > \mu > \beta$, the threshold lies in $(0,1)$.
    \item \emph{Increasing in $\mu$}: stronger deradicalisation raises the
      resilience threshold.
    \item \emph{Dependence on $\beta$}: the sign of
      $\partial\Delta_c/\partial\beta = (\mu-\delta)/(\delta-\beta)^2$
      depends on $\delta$ vs.\ $\mu$.  If $\delta>\mu$: $\Delta_c$ is
      \emph{decreasing} in $\beta$ (pre-existing polarisation lowers the
      threshold).  If $\delta<\mu$: the system is shock-proof for all
      $\Delta\le1$ in this regime anyway.
    \item \emph{Decreasing in $\delta$} (for $\delta>\mu$): higher radical
      mobilisation efficiency lowers the threshold.
    \item $\Delta_c\to0$ as $\beta\to\mu^-$ (when $\delta>\mu$): near the
      bifurcation point, arbitrarily small shocks suffice.
  \end{itemize}
\end{corollary}

\begin{proof}
  All claims follow by direct differentiation of \eqref{eq:Delta_c} and
  sign analysis.  The condition $\Delta_c<1$ is equivalent to $\mu-\beta<\delta-\beta$,
  i.e.\ $\mu<\delta$.   
\end{proof}

The proximity to bifurcation in Corollary~\ref{cor:Dc} is the mathematical
expression of a well-known political-science observation: societies that are
``already polarised'' are disproportionately sensitive to external shocks.
The baseline model predicts a smooth, proportional response; the four-group
model predicts threshold behaviour with an explicit formula.

\subsection{The radicalization window theorem}

\begin{theorem}[Radicalization window upper bound]\label{thm:window}
  Under the conditions of Theorem~\ref{thm:critical} with $\Delta>\Delta_c$
  and $P_0=\varepsilon\ll1$, consider the early phase in which $P(t)$ remains
  small.  In this regime $\dot{A}\approx-\rho A$, giving $A(t)\lesssim\Delta e^{-\rho t}$
  (an upper bound, since the full decay rate $2\delta P+\rho\ge\rho$ exceeds $\rho$
  once $P>0$).

  The instantaneous radical growth rate $\dot{P}/P$, approximated by replacing
  $A(t)$ with its upper bound $\Delta e^{-\rho t}$, changes sign at the
  \emph{estimated window duration}
  \begin{equation}\label{eq:tstar}
    t^* = \frac{1}{\rho}\ln\!\left(\frac{\Delta}{\Delta_c}\right)
        = \frac{1}{\rho}\ln\!\left(
            \frac{(\delta-\beta)\,\Delta}{\mu-\beta}\right).
  \end{equation}
  The actual surge duration $\hat{t}$ satisfies $\hat{t} \le t^*$:
  replacing $A(t)$ by its upper bound overestimates the available mobilisation
  fuel.  For $P_0\approx0$, the bound $t^*$ is asymptotically tight.
\end{theorem}

\begin{proof}
  Since $2\delta P(t)+\rho \ge \rho$ for all $P\ge0$, we have
  \[
    A(t)\le \Delta e^{-\rho t}.
  \]
  The exact instantaneous growth rate is
  \[
    \frac{\dot P}{P}
    = (\beta-\mu)-2\beta P + (\delta-\beta)A(t).
  \]
  Hence
  \[
    \frac{\dot P}{P}
    \le (\beta-\mu)+(\delta-\beta)A(t)
    \le (\beta-\mu)+(\delta-\beta)\Delta e^{-\rho t}.
  \]
  The right-hand side changes sign at the stated value of $t^*$.
  Therefore the true surge duration $\hat t$ satisfies $\hat t\le t^*$.

  For $P_0=\varepsilon\ll1$, the early-phase approximation $P(t)\approx0$
  makes the correction term $-2\beta P$ negligible, while
  $A(t)\approx\Delta e^{-\rho t}$.  In that regime $t^*$ is asymptotically
  tight as an estimate of the true sign-change time.   
\end{proof}

\begin{corollary}[Policy levers on window duration]\label{cor:levers}
  The window bound $t^*$ in \eqref{eq:tstar} depends on model parameters
  as follows.  Writing $t^* = \rho^{-1}\ln(\Delta/\Delta_c)$:
  \begin{enumerate}[label=(\roman*)]
    \item \textbf{Re-engagement rate $\rho$:} enters linearly as $1/\rho$;
      doubling $\rho$ halves $t^*$.  This is the most direct institutional
      lever: investment in civic re-engagement shortens the window proportionally.
    \item \textbf{Shock ratio $\Delta/\Delta_c$:} enters logarithmically;
      reducing $\Delta$ or increasing $\Delta_c$ has diminishing returns.
      
    \item \textbf{Polarisation parameter $\beta$:} enters through $\Delta_c$
      via the ratio $(\delta-\beta)/(\mu-\beta)$.  When $\delta>\mu$, $\beta\to\mu$
      makes $\Delta_c\to0$ and $t^*\to\infty$: a society near the bifurcation
      point has an arbitrarily long radicalization window.  Thus reducing $\beta$
      both raises $\Delta_c$ and shortens $t^*$, making it a lever that acts on
      both the threshold and the window duration simultaneously.
  \end{enumerate}
\end{corollary}

\section{The Bifurcation Trigger and Cumulative Shocks}\label{sec:cumulative}

\subsection{The bifurcation trigger theorem}

\begin{theorem}[Bifurcation trigger]\label{thm:trigger}
  Consider a system initially in Regime~I with parameter $\beta_0<\mu$ and
  with a nonzero radical seed $P(t_0^-)>0$ (equivalently, $L(t_0^-)+R(t_0^-)>0$
  in the full model).  A structural shock of amplitude $(\Delta,\Delta\beta)$ at
  time $t_0$ causes a \emph{permanent shift of the long-run equilibrium} if and
  only if
  \begin{equation}\label{eq:trigger}
    \beta_0 + \Delta\beta > \mu.
  \end{equation}
  When \eqref{eq:trigger} holds, the post-shock long-run centrist share is
  \begin{equation}\label{eq:Cinf}
    C^\infty = \frac{\mu}{\beta_0+\Delta\beta} < 1,
  \end{equation}
  and this value is independent of the state component $\Delta$ of the shock.

  If $P(t_0^-)=0$, then $P(t)\equiv0$ for all $t>t_0$ regardless of
  $\Delta\beta$, and the system converges to $E_0$ (full centrist recovery).
\end{theorem}

\begin{proof}
  The case $P(t_0^-)=0$ is handled by Proposition~\ref{prop:stab4}(ii):
  $P\equiv0$ persists and the system returns to $E_0$.

  For $P(t_0^-)>0$: the state shock preserves $P(t_0^+)=P(t_0^-)>0$ (only
  $C$ and $A$ are affected by the impulse).  Post-shock, the system evolves
  under parameters $(\beta',\mu)$ with $\beta'=\beta_0+\Delta\beta$ from a
  state with $P(t_0^+)>0$.  By Proposition~\ref{prop:stab4}(ii), the
  long-run attractor is $E_0$ if $\beta'\le\mu$ and $E_1$ with
  $C^*=\mu/\beta'$ if $\beta'>\mu$.  Condition \eqref{eq:trigger} is exactly
  $\beta'>\mu$.  The formula \eqref{eq:Cinf} follows from Corollary~\ref{cor:Cstar}
  with $\beta$ replaced by $\beta'$.  Since $C^\infty$ depends only on $\beta'$
  and $\mu$, it is independent of $\Delta$.   
\end{proof}

\begin{remark}[The fundamental decomposition]
  Theorem~\ref{thm:trigger} provides a clean decomposition of any shock
  into two independently consequential components:
  \begin{description}
    \item[State component $\Delta$:] determines whether a surge occurs
      (Theorem~\ref{thm:critical}) and how long it lasts
      (Theorem~\ref{thm:window}).  Does \emph{not} determine the long-run
      equilibrium.  This is the object of crisis management.
    \item[Structural component $\Delta\beta$:] determines whether the
      long-run equilibrium permanently shifts.  The magnitude of $\Delta$
      is irrelevant to this determination.  This is the object of structural
      resilience.
  \end{description}
  A society can successfully manage a crisis (limit $\Delta$) while failing
  to maintain structural resilience (allowing $\Delta\beta>0$), and will
  still end up at a permanently more radicalised equilibrium.
\end{remark}

\begin{corollary}[Irreversibility of structural shifts]\label{cor:irreversible}
  Under the conditions of Theorem~\ref{thm:trigger} with $\beta_0+\Delta\beta>\mu$,
  return to the pre-shock centrist share $C^*_{\mathrm{pre}}=1$ (or to any
  centrist share higher than $C^\infty=\mu/(\beta_0+\Delta\beta)$) is impossible
  under the autonomous dynamics.  Recovery requires a structural counter-shock:
  a sustained reduction in $\alpha+\gamma$ that reverses the parameter shift.
\end{corollary}

\begin{proof}
  By Proposition~\ref{prop:stab4}(ii), every trajectory with $P_0>0$ converges
  to $E_1$ with centrist share $\mu/\beta'$.  Since $\beta'=\beta_0+\Delta\beta>\mu$
  is a fixed parameter, the attractor is unique and independent of initial
  conditions.   
\end{proof}

\begin{remark}[Role of the disengaged compartment]
  The disengaged compartment $A(t)$ serves two distinct functions in the
  analysis.  First, it captures transient crisis-induced withdrawal that
  amplifies the radical surge and determines the window duration $t^*$.
  Second, via the depletion rates $\delta$ and $\rho$, it modulates the
  effective threshold $\Delta_c$.  However, $A(t)$ does not by itself
  generate irreversible long-run change: since $A^*=0$ at every
  equilibrium (Theorem~\ref{thm:eq4}), all asymptotic dynamics lie on
  the face $A=0$ and are governed by the baseline parameters.
  Irreversible long-run radicalization arises exclusively from
  \emph{permanent structural shifts} in $\beta$ that push the system
  across the Perron--Frobenius threshold $\mathcal{R}_{\mathrm{rad}}=1$.
\end{remark}

\begin{remark}[Asymmetric generalisation via Perron--Frobenius]
\label{rem:PF_trigger}
  Theorem~\ref{thm:trigger} is stated on the symmetric invariant manifold
  $L=R=P$.  In the full asymmetric four-group model, the natural generalisation
  of the transcritical condition $\beta=\mu$ is the crossing
  \[
    \lambda_{\mathrm{PF}}\bigl(M_{\mathrm{post}}^{-1}K_{\mathrm{post}}\bigr)=1,
  \]
  where $M_{\mathrm{post}}$ and $K_{\mathrm{post}}$ are the post-shock
  deradicalisation and recruitment matrices (as in
  Theorem~\ref{thm:PF4}).
  When this eigenvalue exceeds~1, the post-shock baseline has an interior
  equilibrium and the system no longer returns to full centrism.  The
  asymmetric transient threshold is therefore spectral rather than scalar;
  Theorem~\ref{thm:PF_Delta} makes this precise, while
  Proposition~\ref{prop:PF_window} provides a comparison bound for the
  corresponding growth window.
\end{remark}

\begin{theorem}[Asymmetric critical shock threshold near the centrist state]
\label{thm:PF_Delta}
  Let
  \[
    K=
    \begin{pmatrix}
      \alpha_L & \gamma_{RL}\\
      \gamma_{LR} & \alpha_R
    \end{pmatrix},
    \qquad
    M=
    \begin{pmatrix}
      \mu_L & 0\\
      0 & \mu_R
    \end{pmatrix},
    \qquad
    D=
    \begin{pmatrix}
      \delta_L & 0\\
      0 & \delta_R
    \end{pmatrix},
  \]
  and define, for $\Delta\in[0,1]$,
  \[
    \Phi(\Delta)
    :=\lambda_{\mathrm{PF}}\!\Bigl(
      M^{-1}\bigl((1-\Delta)K+\Delta D\bigr)
    \Bigr).
  \]
  Assume the baseline is subcritical:
  $\Phi(0)=\lambda_{\mathrm{PF}}(M^{-1}K)<1$.

  Then the following hold.
  \begin{enumerate}[label=(\roman*)]
    \item If $\Phi(1)=\lambda_{\mathrm{PF}}(M^{-1}D)
      =\max\{\delta_L/\mu_L,\delta_R/\mu_R\}\le1$,
      then $\Phi(\Delta)<1$ for all $\Delta\in[0,1]$.
      Hence no feasible pure state shock can produce initial
      near-centrist radical growth.
    \item If $\Phi(1)>1$, then there exists a unique
      $\Delta_c^{\mathrm{asym}}\in(0,1)$ such that
      \[
        \Phi(\Delta_c^{\mathrm{asym}})=1.
      \]
      Consider a pure state shock of amplitude $\Delta$ from a near-centrist
      post-shock state $x(t_0^+)=(L,R)^\top=\varepsilon v$ with $v\gg0$ and
      $0<\varepsilon\ll1$, and let
      \[
        G(\Delta):=(1-\Delta)K+\Delta D-M.
      \]
      If $q(\Delta)\gg0$ is a positive left Perron eigenvector of the
      Metzler matrix $G(\Delta)$, then the weighted radical mass
      $V_\Delta(x):=q(\Delta)^\top x$ satisfies
      \[
        \dot V_\Delta(t_0^+)
        = s\!\bigl(G(\Delta)\bigr)\,V_\Delta(t_0^+) + O(\varepsilon^2),
      \]
      and therefore, for sufficiently small $\varepsilon$,
      \[
        \dot V_\Delta(t_0^+)\gtrless 0
        \iff
        s\!\bigl(G(\Delta)\bigr)\gtrless 0
        \iff
        \Phi(\Delta)\gtrless 1
        \iff
        \Delta\gtrless \Delta_c^{\mathrm{asym}}.
      \]
  \end{enumerate}
  In the symmetric case this threshold reduces exactly to
  \[
    \Delta_c^{\mathrm{asym}}
    = \frac{\mu-\beta}{\delta-\beta}.
  \]
\end{theorem}

\begin{proof}
  For the near-centrist post-shock state one has
  \[
    \dot x = G(\Delta)\,x + O(\varepsilon^2),
    \qquad
    G(\Delta)=(1-\Delta)K+\Delta D-M.
  \]
  If $q(\Delta)\gg0$ is a positive left Perron eigenvector of the irreducible
  Metzler matrix $G(\Delta)$, then
  \[
    \dot V_\Delta(t_0^+)
    = q(\Delta)^\top \dot x
    = s\!\bigl(G(\Delta)\bigr)\,V_\Delta(t_0^+) + O(\varepsilon^2),
  \]
  which proves the sign criterion once the threshold structure of
  $\Phi(\Delta)$ is established.

  By the M-matrix criterion,
  \[
    s\!\bigl(G(\Delta)\bigr)\gtrless 0
    \iff
    \lambda_{\mathrm{PF}}\!\Bigl(
      M^{-1}\bigl((1-\Delta)K+\Delta D\bigr)
    \Bigr)\gtrless 1,
  \]
  so the sign of the instantaneous linearised growth is determined by
  $\Phi(\Delta)-1$.

  Write
  \[
    a_L(\Delta)=\frac{(1-\Delta)\alpha_L+\Delta\delta_L}{\mu_L},
    \quad
    a_R(\Delta)=\frac{(1-\Delta)\alpha_R+\Delta\delta_R}{\mu_R},
    \quad
    b_0=\frac{\gamma_{RL}}{\mu_L},
    \quad
    c_0=\frac{\gamma_{LR}}{\mu_R}.
  \]
  Then
  \[
    \Phi(\Delta)
    = \frac12\Bigl(
        a_L(\Delta)+a_R(\Delta)
        + \sqrt{(a_L(\Delta)-a_R(\Delta))^2
        + 4b_0c_0(1-\Delta)^2}
      \Bigr).
  \]
  The square-root term is the Euclidean norm of the affine vector
  \[
    \Bigl(
      a_L(\Delta)-a_R(\Delta),
      \,2\sqrt{b_0c_0}\,(1-\Delta)
    \Bigr),
  \]
  hence is a convex function of $\Delta$.  Therefore $\Phi$ is continuous
  and convex on $[0,1]$.

  If $\Phi(1)\le1$, then convexity implies
  \[
    \Phi(\Delta)
    \le (1-\Delta)\Phi(0)+\Delta\Phi(1)
    < 1
    \qquad (0\le\Delta<1),
  \]
  and also $\Phi(1)\le1$, proving part~(i).

  If $\Phi(1)>1$, continuity yields existence of at least one root of
  $\Phi(\Delta)=1$ in $(0,1)$.  Because the sublevel set
  $\{\Delta\in[0,1]:\Phi(\Delta)\le1\}$ of a convex function is an
  interval, and because $\Phi(0)<1<\Phi(1)$, this root is unique.
  The symmetric reduction is immediate, since
  \[
    \Phi(\Delta)
    = \frac{(1-\Delta)\beta+\Delta\delta}{\mu}
  \]
  when $\alpha_L=\alpha_R=\alpha$,
  $\gamma_{LR}=\gamma_{RL}=\gamma$,
  $\mu_L=\mu_R=\mu$,
  and $\delta_L=\delta_R=\delta$.   
\end{proof}

\begin{remark}
  Unlike the symmetric scalar formula, $\Phi(\Delta)$ need not be monotone in
  $\Delta$: increasing disengagement both strengthens the diagonal
  mobilisation terms $\delta_iA$ and weakens the centrist-mediated recruitment
  terms through the factor $(1-\Delta)$.  The uniqueness of
  $\Delta_c^{\mathrm{asym}}$ therefore comes from convexity, not from
  monotonicity.
\end{remark}

\begin{remark}[Quadratic equation for $\Delta_c^{\mathrm{asym}}$]
\label{rem:Dc_formula}
  Write
  \[
    a_i:=\alpha_i-\mu_i,
    \qquad
    e_i:=\delta_i-\alpha_i,
    \qquad
    g:=\gamma_{RL}\gamma_{LR},
  \]
  so that
  \[
    G(\Delta)=
    \begin{pmatrix}
      a_L+e_L\Delta & (1-\Delta)\gamma_{RL}\\
      (1-\Delta)\gamma_{LR} & a_R+e_R\Delta
    \end{pmatrix}.
  \]
  At the threshold one has $s(G(\Delta_c^{\mathrm{asym}}))=0$, hence
  \[
    \det G(\Delta_c^{\mathrm{asym}})=0.
  \]
  This yields the quadratic equation
  \begin{equation}\label{eq:Dc_quadratic}
    q_2\,\Delta^2 + q_1\,\Delta + q_0 = 0,
  \end{equation}
  with coefficients
  \[
    q_2 := e_Le_R - g,
    \qquad
    q_1 := a_Le_R + a_Re_L + 2g,
    \qquad
    q_0 := a_La_R - g = \det(K-M) > 0.
  \]
  Equivalently,
  \[
    (a_L+e_L\Delta)(a_R+e_R\Delta)
    = g(1-\Delta)^2.
  \]

  The critical threshold $\Delta_c^{\mathrm{asym}}$ is the
  \emph{first root to the right of $0$}; equivalently, it is the unique root of
  \eqref{eq:Dc_quadratic} in $(0,1)$ satisfying
  \[
    \operatorname{tr} G(\Delta)\le 0.
  \]
  If $q_2\neq0$, this root is
  \begin{equation}\label{eq:Dc_asym_explicit}
    \Delta_c^{\mathrm{asym}}
    = \frac{-q_1-\sqrt{q_1^2-4q_2q_0}}{2q_2},
  \end{equation}
  while if $q_2=0$ it reduces to
  \[
    \Delta_c^{\mathrm{asym}}=-\frac{q_0}{q_1}.
  \]
  A second algebraic root may exist (and may even lie in $(0,1)$); when it
  does, it corresponds to vanishing of the non-Perron eigenvalue rather than to
  the instability threshold.

  In the symmetric case
  ($a_L=a_R=a$, $e_L=e_R=e$, $\gamma_{LR}=\gamma_{RL}=\gamma$),
  \eqref{eq:Dc_quadratic} becomes
  \[
    (a+e\Delta)^2=\gamma^2(1-\Delta)^2.
  \]
  The relevant branch is
  \[
    a+e\Delta=-\gamma(1-\Delta),
  \]
  which gives
  \[
    \Delta_c^{\mathrm{asym}}
    = \frac{\mu-\beta}{\delta-\beta}
    = \Delta_c,
  \]
  recovering the symmetric scalar formula exactly.
\end{remark}

\begin{remark}[A toy asymmetric computation]
  For example, with
  \[
    (\alpha_L,\alpha_R,\gamma_{RL},\gamma_{LR},\mu_L,\mu_R,\delta_L,\delta_R)
    =(0.08,0.14,0.10,0.06,0.30,0.28,0.55,0.60),
  \]
  formula \eqref{eq:Dc_asym_explicit} gives
  \[
    \Delta_c^{\mathrm{asym}} \approx 0.2347.
  \]
  Hence a pure state shock with $\Delta=0.20$ is subcritical, whereas
  $\Delta=0.25$ is supercritical in the sense of
  Theorem~\ref{thm:PF_Delta}.
\end{remark}

\begin{proposition}[Asymmetric Perron-vector window bound]
\label{prop:PF_window}
  Let
  \[
    x(t):=(L(t),R(t))^\top,
    \qquad
    s_0:=s(K-M)<0,
  \]
  and let $q\gg0$ be a positive left Perron eigenvector of the Metzler matrix
  $K-M$, normalised arbitrarily so that
  \[
    q^\top(K-M)=s_0 q^\top.
  \]
  Define the weighted radical mass $V(t):=q^\top x(t)$ and the comparison
  coefficient
  \[
    \bar\kappa
    := \max_{i=1,2}\frac{(q^\top(D-K))_i}{q_i}.
  \]
  For every post-shock trajectory of the full asymmetric four-group model
  \eqref{eq:4group} with $\sigma\equiv0$ arising from a pure state shock of
  amplitude $\Delta$ applied at $t=0$ to a state with $A(0^-)=0$,
  one has
  \[
    \dot V(t)
    \le \bigl[s_0 + \bar\kappa A(t)\bigr]V(t)
    \le \bigl[s_0 + \bar\kappa\Delta e^{-\rho t}\bigr]V(t).
  \]
  Consequently:
  \begin{enumerate}[label=(\roman*)]
    \item If $\bar\kappa\le0$, then $V(t)$ is strictly decreasing for all
      $t\ge0$.
    \item If $\bar\kappa>0$, define
      \[
        \Delta_q := -\frac{s_0}{\bar\kappa},
        \qquad
        t_q^*:=\frac{1}{\rho}\ln\!\left(\frac{\Delta}{\Delta_q}\right)
        \quad (\Delta>\Delta_q).
      \]
      Then $\Delta\le\Delta_q$ implies that $V(t)$ is nonincreasing for all
      $t\ge0$.  If $\Delta>\Delta_q$, every interval on which $V$ is
      increasing is contained in
      \[
        0\le t < t_q^*.
      \]
  \end{enumerate}
  In the symmetric case one may take $q=(1,1)^\top$, for which
  $s_0=\beta-\mu$, $\bar\kappa=\delta-\beta$, and therefore
  \[
    \Delta_q=\frac{\mu-\beta}{\delta-\beta}=\Delta_c,
    \qquad
    t_q^*=\frac{1}{\rho}\ln\!\left(\frac{\Delta}{\Delta_c}\right)=t^*.
  \]
\end{proposition}

\begin{proof}
  Using $C=1-L-R-A$, the $(L,R)$-subsystem of \eqref{eq:4group} can be written
  exactly as
  \[
    \dot x
    = \bigl[(1-A)K + AD - M\bigr]x - (L+R)Kx.
  \]
  Therefore
  \[
    \dot V
    = q^\top \dot x
    = s_0V + A\,q^\top(D-K)x - (L+R)\,q^\top Kx.
  \]
  Since $x\ge0$, $q\gg0$, and $K$ has nonnegative entries,
  $(L+R)\,q^\top Kx\ge0$.  By definition of $\bar\kappa$,
  \[
    q^\top(D-K)x
    \le \bar\kappa\,q^\top x
    = \bar\kappa V.
  \]
  Hence
  \[
    \dot V \le [s_0+\bar\kappa A(t)]V.
  \]
  Moreover,
  \[
    \dot A
    = -(\delta_L L+\delta_R R+\rho)A
    \le -\rho A,
  \]
  so
  \[
    A(t)\le A(0)e^{-\rho t}.
  \]
  By the impulse jump condition \eqref{eq:jump} with $A(0^-)=0$,
  \[
    A(0)=\Delta\,C(0^-)\le\Delta.
  \]
  Hence
  \[
    A(t)\le \Delta e^{-\rho t}.
  \]
  Combining the two inequalities gives the stated comparison bound.

  If $\bar\kappa\le0$, then $s_0+\bar\kappa A(t)<0$ for all $t$, so $V$
  is strictly decreasing.  If $\bar\kappa>0$ and $\Delta\le\Delta_q$, then
  $s_0+\bar\kappa\Delta e^{-\rho t}\le s_0+\bar\kappa\Delta\le0$ for
  all $t$, hence $V$ is nonincreasing.  Finally, if $\Delta>\Delta_q$, then
  the comparison coefficient becomes negative once
  $\Delta e^{-\rho t}<\Delta_q$, i.e.\ for $t>t_q^*$.  Thus $V$ cannot
  increase after time $t_q^*$.  The symmetric reduction follows by direct
  substitution.   
\end{proof}

\subsection{Cumulative shocks}

The European experience since 2008 is not characterised by a single large
shock but by a sequence of distinct crises.

\begin{definition}[Shock sequence]\label{def:sequence}
  A \emph{shock sequence} is a collection of mixed shocks at times
  $t_1 < t_2 < \cdots < t_n$, with state components $\Delta_1,\ldots,\Delta_n$
  and structural components $\Delta\beta_1,\ldots,\Delta\beta_n\ge0$.
  The \emph{cumulative structural shift} after $k$ shocks is
  \[
    B_k := \beta_0 + \sum_{j=1}^k \Delta\beta_j.
  \]
\end{definition}

\begin{theorem}[Cumulative shock theorem]\label{thm:cumulative}
  Consider a system in Regime~I ($\beta_0<\mu$) with a nonzero
  radical seed ($P(t_k^+)>0$ after each shock).
  Let $\{(\Delta_k, \Delta\beta_k)\}$ be a shock sequence with $\Delta\beta_k\ge0$,
  and let $k^*$ be the first index such that $B_{k^*}>\mu$.
  Then:
  \begin{enumerate}[label=(\roman*)]
\item For all $k<k^*$: the long-run centrist share after each shock
  is $C_k^\infty = 1$.  If, in addition, $\delta>B_k$, then near-centrist
  transient surges are possible only when
  \[
    \Delta_k>\Delta_c^{(k)},\qquad
    \Delta_c^{(k)}=\frac{\mu-B_k}{\delta-B_k},
  \]
  where the surge threshold is evaluated under the \emph{post-shock}
  parameter value $B_k$.  If $\delta\le B_k$, the near-centrist threshold
  formula from Theorem~\ref{thm:critical} does not apply in this form.
    \item At $k=k^*$: the system crosses the bifurcation threshold and
      permanently transitions to Regime~II.  The long-run centrist share
      is $C_{k^*}^\infty = \mu/B_{k^*} < 1$.
    \item Each subsequent shock satisfies
      \[
        C_k^\infty = \frac{\mu}{B_k} \le C_{k-1}^\infty,
      \]
      with strict inequality if and only if $\Delta\beta_k>0$.
    \item The threshold-crossing index $k^*$ satisfies
      \begin{equation}\label{eq:kstar}
        k^* = \min\!\left\{k : \sum_{j=1}^k \Delta\beta_j > \mu - \beta_0\right\}.
      \end{equation}
  \end{enumerate}
\end{theorem}

\begin{proof}
  Parts (i)--(iii) follow from Theorem~\ref{thm:trigger} applied
  iteratively: after each shock the post-shock parameter is $B_k$,
  and the long-run attractor is $E_0$ if $B_k<\mu$ and $E_1(B_k)$
  if $B_k>\mu$.  Part (iv) is a restatement of the definition of $k^*$.   
\end{proof}

\begin{corollary}[Staircase dynamics]\label{cor:staircase}
  Under the conditions of Theorem~\ref{thm:cumulative}, after the
  threshold-crossing shock $k^*$, every subsequent shock with $\Delta\beta_k>0$
  strictly decreases the long-run centrist share: $C_{k}^\infty < C_{k-1}^\infty$.
  If all subsequent structural components satisfy $\Delta\beta_k>0$, the
  sequence $C_{k^*}^\infty > C_{k^*+1}^\infty > \cdots$ is strictly decreasing.
  This is the \emph{staircase pattern} discussed qualitatively for Germany
  and France in Section~\ref{sec:empirical}.
\end{corollary}

\begin{proof}
  Immediate from Theorem~\ref{thm:cumulative}(iii): $\Delta\beta_k>0$ implies
  $B_k>B_{k-1}$, hence $\mu/B_k < \mu/B_{k-1}$.   
\end{proof}

\section{Qualitative Illustrations from Germany and France}\label{sec:empirical}

This section is illustrative only and does not identify shocks or estimate
parameters.  Germany is qualitatively consistent with a sequence of mixed
shocks and cumulative structural shifts: AfD support moved from 4.7\%
(2013) to 12.6\% (2017), 10.3\% (2021), and 20.8\% (2025), a pattern
compatible with transient oscillations around a rising post-crisis radical
floor \cite{bundeswahlleiterin2025,eurobarometer2013}.  The 2021 dip is not
an anomaly: Corollary~\ref{cor:staircase} governs long-run floors, not
monotone election-by-election observations.
France is used more cautiously: because presidential, legislative, and
European elections are institutionally different and not commensurate,
the case serves only to motivate the asymmetric setting in which both the
RN and the LFI bloc have strengthened while the mainstream centre weakens
\cite{ministere2024}.  No econometric identification is claimed.

The figures above are Zweitstimmen or vote-share percentages of votes
cast.  Section~\ref{sec:proxies} below uses a different normalisation:
$V=(L+R)\times\text{turnout}$ expressed as a fraction of the eligible
electorate, so that $L+R+C+A=1$ holds exactly.

\subsection{Observable proxies and parameter plausibility}
\label{sec:proxies}

\paragraph{Mapping to observables.}
The conservation law $L+R+C+A=1$ requires that all four compartments
be expressed as fractions of the same reference population.  We use the
\emph{eligible electorate} (all registered voters):
\begin{align}
  R(t) &= \bigl(\text{AfD Zweitstimmen\%}\bigr)\times\text{turnout},
  \notag\\
  L(t) &= \bigl(\text{far-left Zweitstimmen\%}\bigr)\times\text{turnout},
  \notag\\
  C(t) &= \bigl(\text{all remaining parties}\bigr)\times\text{turnout},
  \notag\\
  A(t) &= 1-\text{turnout}.
  \label{eq:proxy}
\end{align}
Here ``far-left'' denotes Die Linke for 2013--2021, and Die Linke
together with BSW for 2025.  All parties not assigned to $L$ or $R$---
CDU/CSU, SPD, FDP, the Greens, and minor parties---are absorbed into $C$.
By construction $L+R+C+A=1$ at every observation, so the conservation
law is satisfied exactly.  Crucially, $A(t)=1-\text{turnout}$ is
directly observable from official election records without any proxy
construction.

\begin{table}[ht]
  \centering
  \caption{Bundestag elections 2013--2025, expressed as fractions of the
    eligible electorate according to the mapping~\eqref{eq:proxy}.
    $V=L+R$ is the total radical share.
    Source: Bundeswahlleiterin~\cite{bundeswahlleiterin2025}.}
  \label{tab:germany}
  \smallskip
  \begin{tabular}{lcccc}
    \hline
    Year & $V=L+R$ & $C$ (mainstream + residual) & $A$ (non-voters)
         & $V+C+A$ \\
    \hline
    2013 & 0.095 & 0.620 & 0.285 & 1.000 \\
    2017 & 0.166 & 0.596 & 0.238 & 1.000 \\
    2021 & 0.116 & 0.650 & 0.234 & 1.000 \\
    2025 & 0.246 & 0.584 & 0.170 & 1.000 \\
    \hline
  \end{tabular}
\end{table}

\paragraph{Background and crisis-induced disengagement.}
The model predicts $A^*=0$ at every equilibrium
(Theorem~\ref{thm:eq4}), meaning that \emph{crisis-induced}
disengagement dissipates in the long run.  The raw observable
$A=1-\text{turnout}$ mixes a structural background of chronic
non-voting with the crisis-induced excess that the model tracks.
The 2025 value $A(2025)=0.170$ is the lowest in the Bundestag series
since reunification and plausibly represents the structural background;
it corresponds to the high-turnout election in which prior disengagement
was largely mobilised into the radical bloc.  Taking $A_{\mathrm{bg}}
\approx 0.170$, the crisis-induced excess $A_{\mathrm{exc}}=A-A_{\mathrm{bg}}$
reads $0.115$ (2013), $0.068$ (2017), $0.064$ (2021), and $0.000$ (2025),
a sequence that declines monotonically toward zero---consistent with
$A^*=0$.  The drop in $A$ from $0.234$ (2021) to $0.170$ (2025) coincides
with the largest single-period increase in $V$, from $0.116$ to $0.246$,
which is the mobilisation signature of the $A\to L\cup R$ mechanism.

\paragraph{Illustrative parameter values.}
We exhibit a parameter set for which the model's equilibrium floors are
consistent with the data.  Fix $\mu=0.22$ and $\beta_0=0.18$,
corresponding to a subcritical baseline ($\mathcal{R}_{\mathrm{rad}}=0.82<1$;
$V^*_0=0$) consistent with the pre-2013 German party system in which no
far-right party had entered the Bundestag since reunification.  After
the 2015 migration crisis, set
\begin{equation}
  \beta_1 = 0.249, \qquad
  \mathcal{R}_{\mathrm{rad}} = \tfrac{\beta_1}{\mu} = 1.13 > 1,\qquad
  V^*_1 = 1 - \tfrac{\mu}{\beta_1} = 0.1165.
  \label{eq:beta1}
\end{equation}
After the 2022 energy and cost-of-living crisis, set
\begin{equation}
  \beta_2 = 0.2917, \qquad
  \mathcal{R}_{\mathrm{rad}} = 1.33 > 1, \qquad
  V^*_2 = 1 - \tfrac{\mu}{\beta_2} = 0.2458.
  \label{eq:beta2}
\end{equation}
The two structural shifts are $\Delta\beta_1=\beta_1-\beta_0=0.069$
(2015) and $\Delta\beta_2=\beta_2-\beta_1=0.043$ (2022).

\begin{remark}[No circularity claim]
  The values of $\beta_1$ and $\beta_2$ in~\eqref{eq:beta1}--\eqref{eq:beta2}
  are \emph{defined} by requiring $V^*_1\approx V(2021)$ and
  $V^*_2\approx V(2025)$.  These two equalities carry no independent
  predictive content.  The non-trivial empirical content of the model
  consists solely of the three structural predictions listed below, none
  of which is used in choosing the parameters.
\end{remark}

\paragraph{Three qualitative predictions.}
Under the calibration~\eqref{eq:beta1}--\eqref{eq:beta2}, the model
makes three structural predictions that can be checked directly against
Table~\ref{tab:germany}.

\begin{enumerate}[label=(P\arabic*)]

\item \emph{Staircase of long-run floors.}
  Corollary~\ref{cor:staircase} predicts a strictly increasing sequence
  of post-shock equilibria: $V^*_0=0 < V^*_1 < V^*_2$.  The relevant
  comparison in the data is between elections that occur well after each
  shock has passed: 2021 (six years after shock~1) and 2025 (three years
  after shock~2).  The observed sequence $V(2021)=0.116 < V(2025)=0.246$
  is consistent with the staircase prediction.

\item \emph{Transient overshoot above the new floor.}
  Theorem~\ref{thm:critical} predicts that immediately after shock~1
  the system overshoots $V^*_1=0.116$ before returning to it.  The 2017
  observation $V(2017)=0.166 > V^*_1=0.116$ is consistent with the
  system being in the transient overshoot phase; by 2021, $V=0.116
  \approx V^*_1$.  The 2021 dip relative to 2017 is thus a
  \emph{model prediction}, not an anomaly requiring separate explanation.

\item \emph{Co-movement of $V$ and $A$.}
  The model predicts that $V$ and $A$ move in opposite directions as
  the crisis-induced pool is absorbed into the radical blocs
  ($A\to L\cup R$).  In the data, $A$ falls monotonically from $0.285$
  (2013) to $0.170$ (2025), while $V$ rises from $0.095$ to $0.246$.
  The largest single-period co-movement---$A$ drops by $0.064$ and $V$
  rises by $0.130$ between 2021 and 2025---is consistent with the
  $\delta$-mobilisation mechanism amplifying the post-2022 radical surge.

\end{enumerate}

We emphasise that this is a \emph{stylised calibration}: the parameter
values are chosen to reproduce the qualitative regime structure
(subcritical pre-2015, supercritical post-2015, staircase after 2022),
not to minimise any statistical criterion.  A rigorous estimation
procedure would require a probabilistic observation model and is beyond
the scope of the present paper.

\section{Numerical Illustrations of the Four-Group Model}\label{sec:numerics4}

All simulations use the \texttt{solve\_ivp} function (RK45 solver) from
\texttt{SciPy} \cite{scipy2020} with relative tolerance $10^{-9}$ and absolute
tolerance $10^{-12}$.

\subsection{Three dynamic regimes}

Figure~\ref{fig:regimes} illustrates the three qualitatively distinct
post-shock trajectories established analytically in Section~\ref{sec:analysis4},
using parameters $\beta_0=0.30$, $\mu=0.40$, $\delta=0.70$, $\rho=0.10$,
giving $\Delta_c=0.25$.

\begin{figure}[ht]
\centering
\includegraphics[width=\textwidth]{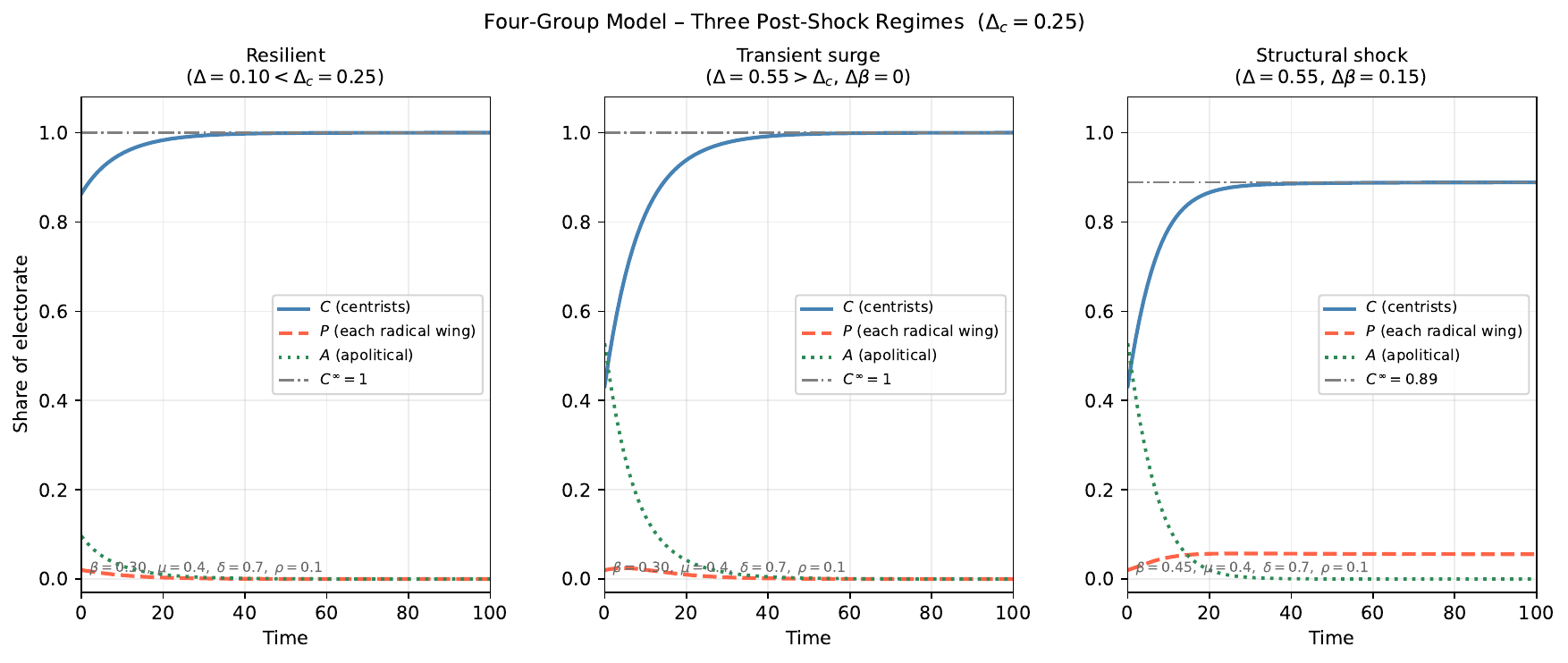}
\caption{The three shock regimes of the four-group model.
  \textbf{Left:} Small shock ($\Delta=0.10 < \Delta_c=0.25$):
  no surge, monotone recovery to $C^\infty=1$.
  \textbf{Centre:} Large state shock ($\Delta=0.55>\Delta_c$, $\Delta\beta=0$):
  transient populist surge followed by full recovery.
  The apolitical pool $A$ (dotted) drives the surge and is depleted on
  timescale $1/\rho=10$.
  \textbf{Right:} Structural shock ($\Delta=0.55$, $\Delta\beta=0.15$):
  $\beta'=0.45>\mu=0.40$; the system converges to a new equilibrium with
  $C^\infty=\mu/\beta'\approx0.89<1$.  The state component $\Delta$ determines
  the transient; the structural component $\Delta\beta$ alone determines $C^\infty$.
  Parameters: $\beta_0=0.30$, $\mu=0.40$, $\delta=0.70$, $\rho=0.10$.}
\label{fig:regimes}
\end{figure}

\subsection{The staircase pattern}

Figure~\ref{fig:staircase} simulates four shocks at times $t=10, 25, 40, 60$
(inter-shock intervals 15, 15, 20, chosen unevenly to illustrate that the
staircase pattern does not require periodic forcing), each with state component
$\Delta_k=0.40$ and structural component $\Delta\beta_k=0.04$.  The cumulative
structural drift $B_k$ crosses $\mu=0.40$ at shock $k^*=3$
(since $\beta_0+3\times0.04=0.42>\mu$).

\begin{figure}[ht]
\centering
\includegraphics[width=0.85\textwidth]{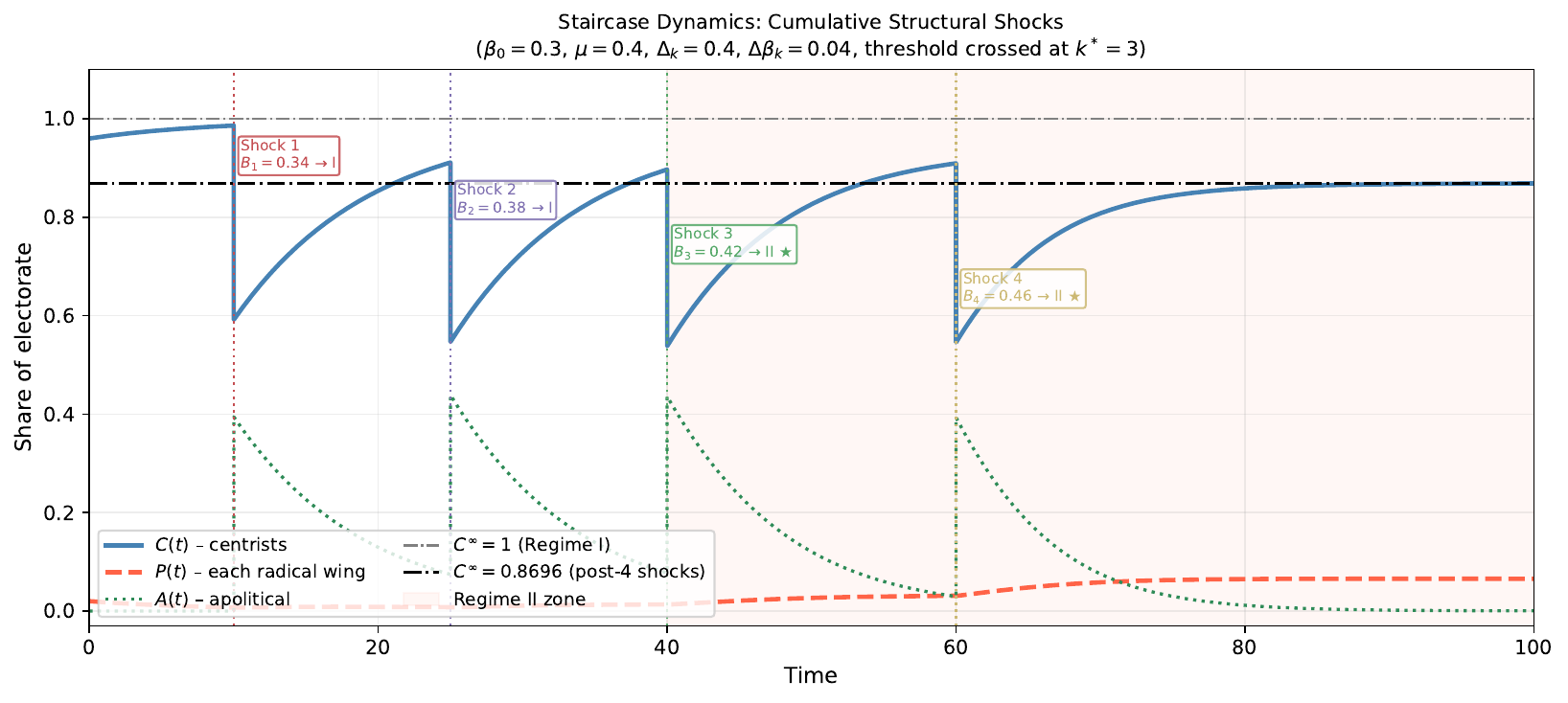}
\caption{Staircase dynamics under four sequential shocks.  Each shock
  has state component $\Delta_k=0.40$ and structural component $\Delta\beta_k=0.04$.
  The cumulative structural shift $B_k$ crosses $\mu=0.40$ at shock $k^*=3$
  (since $\beta_0+3\times0.04=0.42>\mu$).  Before $k^*$: the centrist share
  recovers (partially) after each shock.  After $k^*$: each additional shock
  permanently lowers the long-run centrist share, producing the staircase
  pattern.  Parameters: $\beta_0=0.30$, $\mu=0.40$, $\delta=0.70$, $\rho=0.10$.}
\label{fig:staircase}
\end{figure}

\subsection{Asymmetric four-group dynamics}\label{sec:numerics_asym}

The two subsections above use symmetric parameters ($\alpha_L=\alpha_R$,
$\delta_L=\delta_R$) to maximise analytical tractability.  We now complement
those illustrations with asymmetric scenarios that exhibit the full
structure guaranteed by Theorems~\ref{thm:PF4} and~\ref{thm:PF_Delta}.

\paragraph{Trajectories below and above the Perron--Frobenius threshold.}
Figure~\ref{fig:asym_traj} shows trajectories of the four-group
system~\eqref{eq:4group} for two asymmetric parameter sets.
In the subcritical case ($\mathcal{R}_{\mathrm{rad}}=0.840<1$),
the impulse shock produces a transient surge in $A(t)$ and a brief
slowing of centrist recovery, but all trajectories converge to $E_0$
($C^\infty=1$) as guaranteed by Theorem~\ref{thm:global4}.
In the supercritical case ($\mathcal{R}_{\mathrm{rad}}=1.593>1$),
the system converges to the asymmetric coexistence equilibrium
$E_1=(L^*,R^*,0)$ with $L^*=0.285\ne R^*=0.088$, determined by the
Perron eigenvector of $M^{-1}K$ (Theorem~\ref{thm:PF4}(ii)).
The centrist equilibrium satisfies $C^*=1/\mathcal{R}_{\mathrm{rad}}=0.628$,
confirmed numerically to within tolerance~$10^{-8}$.

\begin{figure}[ht]
\centering
\includegraphics[width=\textwidth]{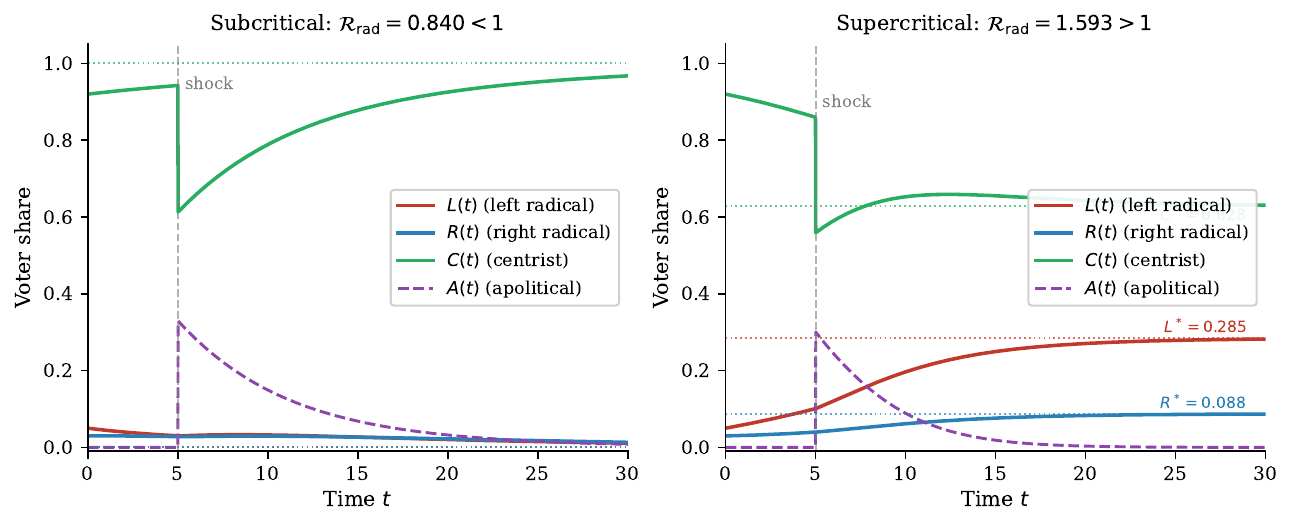}
\caption{Four-group asymmetric trajectories with an impulse shock of
  amplitude $\Delta=0.35$ at $t=5$.
  \textbf{Left} ($\mathcal{R}_{\mathrm{rad}}=0.840<1$):
    subcritical regime.  The shock produces a transient apolitical
    surge ($A$, dashed purple) and brief deceleration of centrist
    recovery, but the system returns to full centrism ($C^\infty=1$,
    dotted green line).
  \textbf{Right} ($\mathcal{R}_{\mathrm{rad}}=1.593>1$):
    supercritical regime.  The system converges to the asymmetric
    interior equilibrium $L^*=0.285\ne R^*=0.088$ (dotted lines),
    with $C^*=1/\mathcal{R}_{\mathrm{rad}}=0.628$.
    The strong left--right asymmetry arises from $\alpha_L>\alpha_R$
    and is encoded in the Perron eigenvector of $M^{-1}K$.
  Parameters: left panel $\alpha_L=0.15$, $\alpha_R=0.20$,
    $\gamma_{RL}=0.08$, $\gamma_{LR}=0.12$, $\delta_L=0.70$,
    $\delta_R=0.55$, $\mu_L=0.30$, $\mu_R=0.35$, $\rho=0.12$;
    right panel $\alpha_L=0.40$, $\alpha_R=0.25$, $\gamma_{RL}=0.15$,
    $\gamma_{LR}=0.08$, $\delta_L=0.75$, $\delta_R=0.60$, $\mu_L=0.28$,
    $\mu_R=0.32$, $\rho=0.10$.}
\label{fig:asym_traj}
\end{figure}

\paragraph{The asymmetric shock threshold $\Delta_c^{\mathrm{asym}}$.}
Figure~\ref{fig:phi_delta} plots $\Phi(\Delta):=
\lambda_{\mathrm{PF}}(M^{-1}((1-\Delta)K+\Delta D))$ as a function of
shock amplitude~$\Delta$ for both a symmetric and an asymmetric parameter
set, both satisfying $\Phi(0)<1$ (subcritical baseline) and $\Phi(1)>1$
(large-shock mobilisation dominant).
Each curve is strictly increasing and crosses $\Phi=1$ exactly once,
at $\Delta_c^{\mathrm{sym}}=0.091$ and $\Delta_c^{\mathrm{asym}}=0.116$
respectively, confirming the uniqueness guaranteed by
Theorem~\ref{thm:PF_Delta}(ii).
The difference between the two thresholds is determined solely by the
asymmetry of $K$ and $D$ relative to $M$; in the symmetric reduction
both thresholds coincide with the scalar formula
$\Delta_c=(\mu-\beta)/(\delta-\beta)$.

\begin{figure}[ht]
\centering
\includegraphics[width=0.65\textwidth]{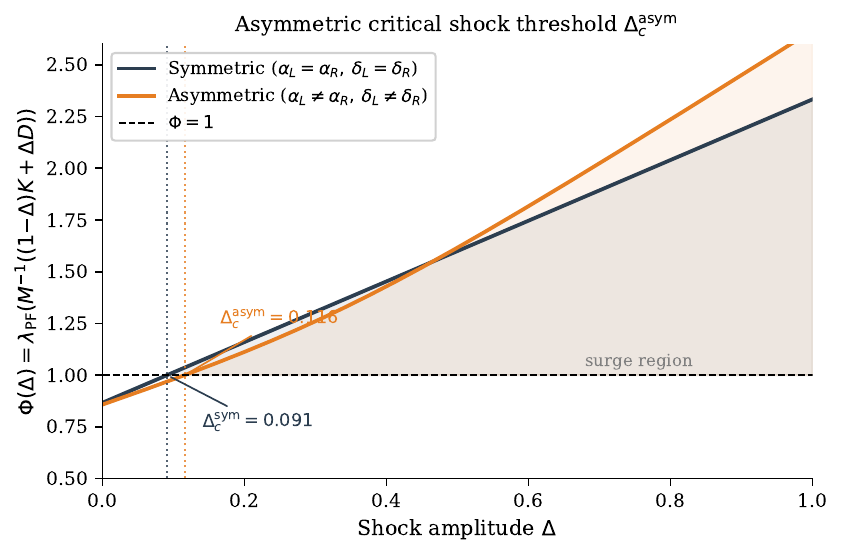}
\caption{The function $\Phi(\Delta)=\lambda_{\mathrm{PF}}
  (M^{-1}((1-\Delta)K+\Delta D))$ for symmetric (dark, $\alpha_L=\alpha_R$,
  $\delta_L=\delta_R$) and asymmetric (orange, $\alpha_L\ne\alpha_R$,
  $\delta_L\ne\delta_R$) parameter sets.  Both curves are strictly
  increasing and cross the threshold $\Phi=1$ exactly once, at
  $\Delta_c^{\mathrm{sym}}=0.091$ and $\Delta_c^{\mathrm{asym}}=0.116$
  (vertical dotted lines).  The shaded region ($\Phi>1$) is the surge
  regime.  The gap between the two thresholds reflects the
  Perron-vector correction of Theorem~\ref{thm:PF_Delta}.
  Symmetric parameters: $\alpha=0.18$, $\gamma=0.08$, $\mu=0.30$,
  $\delta=0.70$.  Asymmetric parameters: $\alpha_L=0.15$,
  $\alpha_R=0.22$, $\gamma_{RL}=0.08$, $\gamma_{LR}=0.05$,
  $\mu_L=\mu_R=0.30$, $\delta_L=0.80$, $\delta_R=0.55$.}
\label{fig:phi_delta}
\end{figure}

\section{Conclusion}\label{sec:conclusion}
\subsection{Summary of main results}

We developed and analysed two coupled ODE models of political competition with
a conserved electorate.

The baseline three-group model is proved to be structurally self-stabilising:
global asymptotic stability holds in both the symmetric and asymmetric cases,
periodic orbits are unconditionally excluded via the Dulac function
$B=(LRC)^{-1}$, and with fixed parameters the model cannot generate staircase
dynamics, history-dependent long-run floors, or multiple attractors.

The four-group extension adds a disengaged compartment and distinguishes pure
state shocks from permanent structural parameter shifts.  The post-shock
dynamics admit a complete global classification governed by the spectral
threshold $\mathcal{R}_{\mathrm{rad}}=\lambda_{\mathrm{PF}}(M^{-1}K)$:
below it the centrist state is globally asymptotically stable; above it every
trajectory with a nonzero radical seed converges to the unique radicalised
equilibrium.  Irreversible long-run shifts arise exclusively from structural
parameter changes crossing this threshold; cumulative sub-threshold shifts
produce staircase dynamics absent from the baseline.  The empirical
illustrations against German Bundestag data 2013--2025 are qualitative only
and do not constitute a statistical test.

\subsection{Identifiability and estimation}\label{sec:identifiability}

A rigorous quantitative calibration would require a probabilistic
observation model linking the ODE state to electoral and survey data,
and faces the following identifiability obstacles.

\paragraph{Structural non-identifiability in the baseline model.}
In the symmetric three-group model, $\alpha$ and $\gamma$ enter the
linearisation only through their sum $\beta=\alpha+\gamma$.  The
Perron--Frobenius threshold $\mathcal{R}_{\mathrm{rad}}=\beta/\mu$ and
the equilibrium $V^*=1-\mu/\beta$ are identified from vote-share time
series, but $\alpha$, $\gamma$, and $\mu$ individually are not.
Separate identification of these parameters would require data on
voter \emph{flows} between consecutive elections, available in
principle from panel surveys such as the German Longitudinal
Election Study (GLES)~\cite{gles2021}.

\paragraph{Four-group identifiability.}
The four-group model adds the mobilisation rate $\delta$ and
re-engagement rate $\rho$.  Because $A(t)=1-\text{turnout}(t)$
is directly observable (Section~\ref{sec:proxies}), the $A$-compartment
provides an independent observation equation, which in principle
identifies $\delta$ and $\rho$ jointly from the co-movement of
vote shares and turnout across elections.  The co-movement prediction
(P3) in Section~\ref{sec:proxies} is an untested version of this
identification: the sign and approximate magnitude of $\Delta A /
\Delta V$ constrains $\delta/\rho$.  Precise estimation, however,
requires a full stochastic formulation that we leave to future work.

\paragraph{Shock decomposition.}
Each shock has a state component $\Delta_k$ (the fraction of $C$
immediately transferred to $A$) and a structural component $\Delta\beta_k$
(the permanent increase in $\beta$).  The state component produces a
turnout dip at the following election, while the structural component
is reflected in the post-shock long-run floor $V^*_k$.  Disentangling
the two components within a single election cycle requires strong
timing assumptions or auxiliary data on between-election survey attitudes.

\paragraph{Limitations and future work.}\label{sec:open}
On the mathematical side, the main open problem is a sharper transient theory
for the asymmetric four-group model, including explicit bounds on surge
duration and peak size under repeated asymmetric shocks, and a more explicit
global theory for structural-shock sequences with parameter drift.  Stochastic
forcing and time-periodic shocks would further allow the study of threshold
crossing and ratchet effects outside the autonomous setting.

The main methodological message is that reversible surges and irreversible
political shifts are distinct dynamical phenomena and should not be described
by the same model: the baseline is appropriate when crises are transient,
the four-group extension when data suggest a ratchet in the long-run radical
floor.



\begin{thebibliography}{99}

\bibitem{abels2024}
C.~M.~Abels \textit{et al.},
``Dodging the autocratic bullet: Enlisting behavioural science to arrest
democratic backsliding,''
\textit{Behavioural Public Policy}, pp.~1--28, 2024.


\bibitem{baumann2020}
F.~Baumann, P.~Lorenz-Spreen, I.~M.~Sokolov, and M.~Starnini,
``Modeling echo chambers and polarization dynamics in social networks,''
\textit{Physical Review Letters}, vol.~124, 048301, 2020.

\bibitem{branstetter2024}
R.~Branstetter, S.~Chian, J.~Cromp, W.~L.~He, C.~M.~Lee, M.~Liu,
E.~Mansell, M.~Paranjape, T.~Pattanashetty, A.~Rodrigues,
and A.~Volkening,
``How time and pollster history affect U.S.\ election forecasts under a
compartmental modeling approach,''
\textit{SIAM J. Appl. Dyn. Syst.}, vol.~25, no.~1, pp.~160--195, 2026,
doi:10.1137/24M1719505.

\bibitem{bundeswahlleiterin2025}
Der Bundeswahlleiter,
``Ergebnisse der Bundestagswahl 2025,''
Federal Returning Officer, Germany, 2025.
Available: \url{https://www.bundeswahlleiterin.de}

\bibitem{dalton2004}
R.~J.~Dalton,
\textit{Democratic Challenges, Democratic Choices}.
Oxford University Press, 2004.

\bibitem{dalton2021}
R.~J.~Dalton,
``Party polarization and citizen attitudes,''
\textit{Political Behavior}, vol.~43, pp.~921--939, 2021.

\bibitem{deffuant2000}
G.~Deffuant, D.~Neau, F.~Amblard, and G.~Weisbuch,
``Mixing beliefs among interacting agents,''
\textit{Advances in Complex Systems}, vol.~3, pp.~87--98, 2000.

\bibitem{diekmann1990}
O.~Diekmann, J.~A.~P. Heesterbeek, and J.~A.~J. Metz.
\newblock On the definition and the computation of the basic reproduction ratio
  $R_0$ in models for infectious diseases in heterogeneous populations.
\newblock {\em Journal of Mathematical Biology}, 28(4):365--382, 1990.

\bibitem{diep2024}
H.~T.~Diep, M.~Kaufman, and S.~Kaufman,
``Application of the three-group model to the 2024 US elections,''
\textit{Entropy}, vol.~27, no.~9, art.~935, 2025.

\bibitem{vandendriessche2002}
P.~van den Driessche and J.~Watmough.
\newblock Reproduction numbers and sub-threshold endemic equilibria for
  compartmental models of disease transmission.
\newblock {\em Mathematical Biosciences}, 180:29--48, 2002.

\bibitem{eurobarometer2013}
European Commission,
``Standard Eurobarometer 80: Public Opinion in the European Union,''
Autumn 2013.

\bibitem{gles2021}
S.~Ro{\ss}teutscher, M.~Debus, T.~Faas, and H.~Schoen,
``GLES Cross-Section 2021, Pre- and Post-Election,''
GESIS Data Archive, Cologne. ZA7702 Version~2.1.0, 2023,
doi:10.4232/1.14170.

\bibitem{grillo2024}
E.~Grillo, Z.~Luo, M.~Nalepa, and C.~Prato,
``Theories of democratic backsliding,''
\textit{Annual Review of Political Science}, vol.~27, pp.~381--400, 2024.

\bibitem{haggard2021}
S.~Haggard and R.~Kaufman,
\textit{Backsliding: Democratic Regress in the Contemporary World}.
Cambridge University Press, 2021.

\bibitem{hegselmann2002}
R.~Hegselmann and U.~Krause,
``Opinion dynamics and bounded confidence: models, analysis, and simulation,''
\textit{Journal of Artificial Societies and Social Simulation},
vol.~5, no.~3, 2002.

\bibitem{hirsch1988}
M.~W.~Hirsch,
``Systems of differential equations that are competitive or cooperative.
III.~Competing species,''
\textit{Nonlinearity}, vol.~1, no.~1, pp.~51--71, 1988.

\bibitem{levitsky2018}
S.~Levitsky and D.~Ziblatt,
\textit{How Democracies Die}.
Crown Publishers, 2018.

\bibitem{lorenz2007}
J.~Lorenz,
``Continuous opinion dynamics under bounded confidence: A survey,''
\textit{International Journal of Modern Physics C}, vol.~18, pp.~1819--1838,
2007.


\bibitem{ministere2024}
Ministère de l'Intérieur,
``Résultats des élections législatives 2024,''
République française, 2024.
Available: \url{https://www.resultats-elections.interieur.gouv.fr}

\bibitem{mudde2019}
C.~Mudde,
\textit{Populist Radical Right Parties in Europe}.
Cambridge University Press, 2019.

\bibitem{norris2019}
P.~Norris and R.~Inglehart,
\textit{Cultural Backlash: Trump, Brexit, and Authoritarian Populism}.
Cambridge University Press, 2019.

\bibitem{perko2001}
L.~Perko,
\textit{Differential Equations and Dynamical Systems}, 3rd~ed.
Springer, 2001.

\bibitem{pew2014}
Pew Research Center,
``Political polarization in the American public,'' 2014.

\bibitem{scipy2020}
P.~Virtanen \textit{et al.},
``SciPy 1.0: Fundamental algorithms for scientific computing in Python,''
\textit{Nature Methods}, vol.~17, pp.~261--272, 2020.

\bibitem{strogatz1994}
S.~H.~Strogatz,
\textit{Nonlinear Dynamics and Chaos}.
Perseus Books, 1994.

\bibitem{torgler2024}
B.~Torgler,
``Dynamics of democratic backsliding due to polarization,''
BITA Working Paper No.~6, Queensland University of Technology, 2024.

\bibitem{yang2020}
J.~Yang \textit{et al.},
``Satisficing dynamics in political competition,''
\textit{SIAM Review}, vol.~64, no.~1, pp.~168--204, 2022.

\bibitem{volkening2020}
A.~Volkening, D.~F.~Linder, M.~A.~Porter, and G.~A.~Rempala,
``Forecasting elections using compartmental models of infection,''
\textit{SIAM Review}, vol.~62, no.~4, pp.~837--865, 2020.

\bibitem{zeeman1993}
E.~C.~Zeeman,
``Hopf bifurcations in competitive three-dimensional Lotka--Volterra systems,''
\textit{Dynamics and Stability of Systems}, vol.~8, no.~3, pp.~189--216, 1993.

\end{thebibliography}
\end{document}